\documentclass{iopjournal}

\usepackage{amsmath,amsthm,amssymb,amsfonts}
\usepackage{xcolor}
\usepackage{graphicx}
\usepackage{mdframed}
\usepackage{mathrsfs}
\usepackage{pifont}
\usepackage{indentfirst}
\usepackage{float}
\usepackage{bbm}
\usepackage{svg}
\usepackage{lipsum}
\usepackage{cite}
\usepackage{algorithm}
\usepackage{algpseudocode}
\usepackage{hyperref}
\usepackage{comment}
\usepackage{tcolorbox}
\hypersetup{
    colorlinks=true,
    linkcolor=blue,
    filecolor=magenta,      
    urlcolor=blue,
    citecolor=blue
}
\usepackage{ragged2e}
\justifying

\graphicspath{{Images/}}

\newcommand{\GG}{\mathcal{G}}

\newcommand{\CC}{\mathcal{C}}

\newcommand{\NN}{\mathcal{N}}

\newcommand{\deltao}{\delta \Omega}
\newcommand{\deltap}{\delta p}

\newcommand{\ubar}{\bar{u}}

\newtheorem{theorem}{Theorem}[section]
\newtheorem{proposition}{Proposition}[section]
\newtheorem{assumption}{Assumption}[section]
\newtheorem{definition}{Definition}[section]

\newtheorem{remark}{Remark}[section]
\newtheorem{lemma}{Lemma}[section]

\AtBeginDocument{%
  \setlength{\textwidth}{16.2cm}
  \setlength{\oddsidemargin}{1cm}
}

\usepackage{etoolbox}

\makeatletter
\gdef\@oddhead{ \hfil\thepage}
\gdef\@evenhead{\thepage\hfil  }
\makeatother

\begin{document}

\title{A variational framework for Bayesian inversion in moving boundary Darcy flow}

\author{M.E. Causon\orcid{0009-0009-3261-0800}, M.A. Iglesias\orcid{0000-0002-8952-717X} and M.V. Tretyakov\orcid{0000-0002-7929-9046}}

\affil{School of Mathematical Sciences, University of Nottingham, Nottingham, UK}

\email{marco.iglesias@nottingham.ac.uk}

\begin{abstract}
\justifying \fontsize{10}{11.5}\selectfont

We develop an infinite-dimensional Bayesian framework for recovering spatially varying log-permeability from pressure observations in single-phase Darcy flow with a moving boundary, motivated by resin transfer moulding. The principal contribution is an observation-wise representer formulation that connects sensitivities of the parameter-to-observable map to the underlying moving-boundary state and adjoint systems. This structure yields a reduced Levenberg--Marquardt method for maximum a posteriori (MAP) estimation on the Cameron--Martin space of a Gaussian prior and an explicit finite-rank representation of the covariance associated with the linearisation around the MAP (LMAP) approximation. Once the representers have been computed, the LM update and covariance
assembly reduce to linear algebra in the observation space,
independently of the discretisation dimension of the permeability
field.

In 1D, the explicit solution of the moving-boundary problem is used to
establish posterior well-posedness and existence of MAP estimators, and
to derive rigorous closed-form expressions for the Fr\'echet derivatives
and their representers. In 2D, we introduce a very weak mixed
formulation in which the pressure, interface velocity, and evolving
domain are treated as coupled variables. A formal one-sided directional
shape linearisation yields the corresponding linearised state and
adjoint systems. Under an explicit differentiability assumption, these
systems provide computable observation-wise representers through linear
adjoint problems that can be solved independently and in parallel.

Numerical experiments show that, in 1D, the explicit
analytical structure reduces the computational cost to
\(\mathcal{O}(10^1)\) forward solves while producing an LMAP approximation
that closely matches a reference Markov chain Monte Carlo (MCMC)
posterior, whose computation requires
\(\mathcal{O}(10^5\text{--}10^6)\) forward simulations. In 2D, the observation-wise adjoint and representer formulation yields
a tractable, parallelisable algorithm. The resulting LMAP approximations produce
reconstructions and uncertainty estimates comparable to ensemble Kalman
inversion (EKI) with large ensemble sizes, at substantially reduced
computational cost, with solutions obtained in minutes rather than hours. The numerical inversion workflow is demonstrated across different
geometries and experimental configurations without requiring
reformulation of the inversion methodology.

\end{abstract}


\section{Introduction}

Moving boundary problems arise when the spatial domain on which a system
of partial differential equations (PDEs) is posed evolves in time. In
such problems, the location of (part of) the boundary is not known
\emph{a priori}, but must instead be determined as part of the solution
itself. These are commonly referred to as free boundary problems and arise
in a wide range of applications, including phase transition phenomena,
subsurface flow, filtration through a porous dam, composites manufacturing, tumour growth, and wound healing
\cite{Stefan,Advani,Long,Dagan,Baiocchi1972,wound}. In many practical settings, the parameters
governing such PDE models are unknown or uncertain and must therefore be
inferred from indirect and often incomplete observations of the system.
This leads naturally to inverse problems in which one seeks to estimate
model parameters from observational data generated by an underlying
physical process described by the PDE system. The presence of a moving
boundary introduces additional complexity, since the unknown parameters
influence not only the solution of the governing PDEs but also the
geometry of the evolving domain itself.

In this work we study a Bayesian inverse problem arising from a moving
boundary model describing the resin infusion stage of Resin Transfer
Moulding (RTM), a manufacturing process used in the production of
fibre-reinforced composite materials \cite{Advani,Long}. During resin infusion, a viscous
resin is injected into a fibrous preform occupying an open bounded domain
\(D\subset\mathbb R^d\), \(d\in\{1,2,3\}\). Its boundary is decomposed as
\(
\partial D
=
\partial D_I
\cup
\partial D_N
\cup
\partial D_0,
\)
where \(\partial D_I\) denotes the inlet boundary,
\(\partial D_0\) the outlet boundary, and \(\partial D_N\) a perfectly
sealed boundary. The fibrous reinforcement is modelled as a porous medium
whose porosity and permeability are, in general, represented by the
spatially varying fields \(\phi(x)\) and \(K(x)\), respectively, for
\(x\in D\). Initially, the reinforcement is filled with air at pressure
\(p_0\). An incompressible resin of viscosity \(\mu_f\) is then injected
through the inlet boundary \(\partial D_I\) at pressure \(p_I>p_0\). At
each time \(t>0\), the resin occupies a time-dependent saturated region
\(\Omega(t)\subseteq D\), whose boundary $\Gamma(t):=\partial\Omega(t)$ consists of the moving resin front
\[
\Upsilon(t):=\Gamma(t)\cap D,
\]
together with the relevant portions of
\(\partial D_I\), \(\partial D_N\), and \(\partial D_0\). An example of
the setting considered here is shown in Figure~\ref{fig: setup_eg}.

Under the assumptions of Darcy flow and incompressibility, the pressure
field \(p\) satisfies
\begin{align}\label{eq:Darcy0}
    -\nabla \cdot \big( e^{u(x)} \nabla p(x,t) \big) = 0,
    \qquad x \in \Omega(t),\quad t>0,
\end{align}
where
\(
u(x):=\log K(x)
\)
denotes the log-permeability field. The motion of the moving front is
governed by the kinematic condition
\cite{Advani,MichaelMinho,Tartakovsky}
\begin{align}\label{eq:Darcy1}
    V(x,t)
    =
    -\frac{1}{\mu_f\phi(x)}
    e^{u(x)}
    \nabla p(x,t)\cdot n(x,t),
    \qquad
    x\in\Upsilon(t),\quad t>0,
\end{align}
where \(V(x,t)\) is the normal velocity of the front and
\(n(x,t)\) denotes the unit outward normal along \(\Upsilon(t)\). The
system is completed with the boundary and initial conditions
\begin{align}
    p(x,t) &= p_I,
    &&x\in\partial D_I,\quad t\geq0,   \label{eq:Darcy2E}\\
    p(x,t) &= p_0,
    &&x\in(\partial D_0\cup\Upsilon(t)),\quad t>0, \notag \\ 
    \nabla p(x,t)\cdot n(x) &= 0,
    &&x\in\partial D_N,\quad t\geq0, \notag \\ 
    p(x,0) &= p_0,
    &&x\in D, \notag \\ 
    \Upsilon(0) &= \partial D_I. \notag 
\end{align}
Here, \(n(x)\) denotes the unit outward normal along the fixed no-flow
boundary \(\partial D_N\).

Throughout this work it is assumed that the physical parameters
\((D,\phi,p_I,p_0,\mu_f)\) are known, although the framework can be extended
to account for uncertainty in some of these quantities. Hence, the system
\eqref{eq:Darcy0}--\eqref{eq:Darcy2E} defines a nonlinear forward operator
\[
u
\mapsto
(p,\Upsilon).
\]

\begin{figure}
    \centering
    \includegraphics[width=0.6\linewidth]{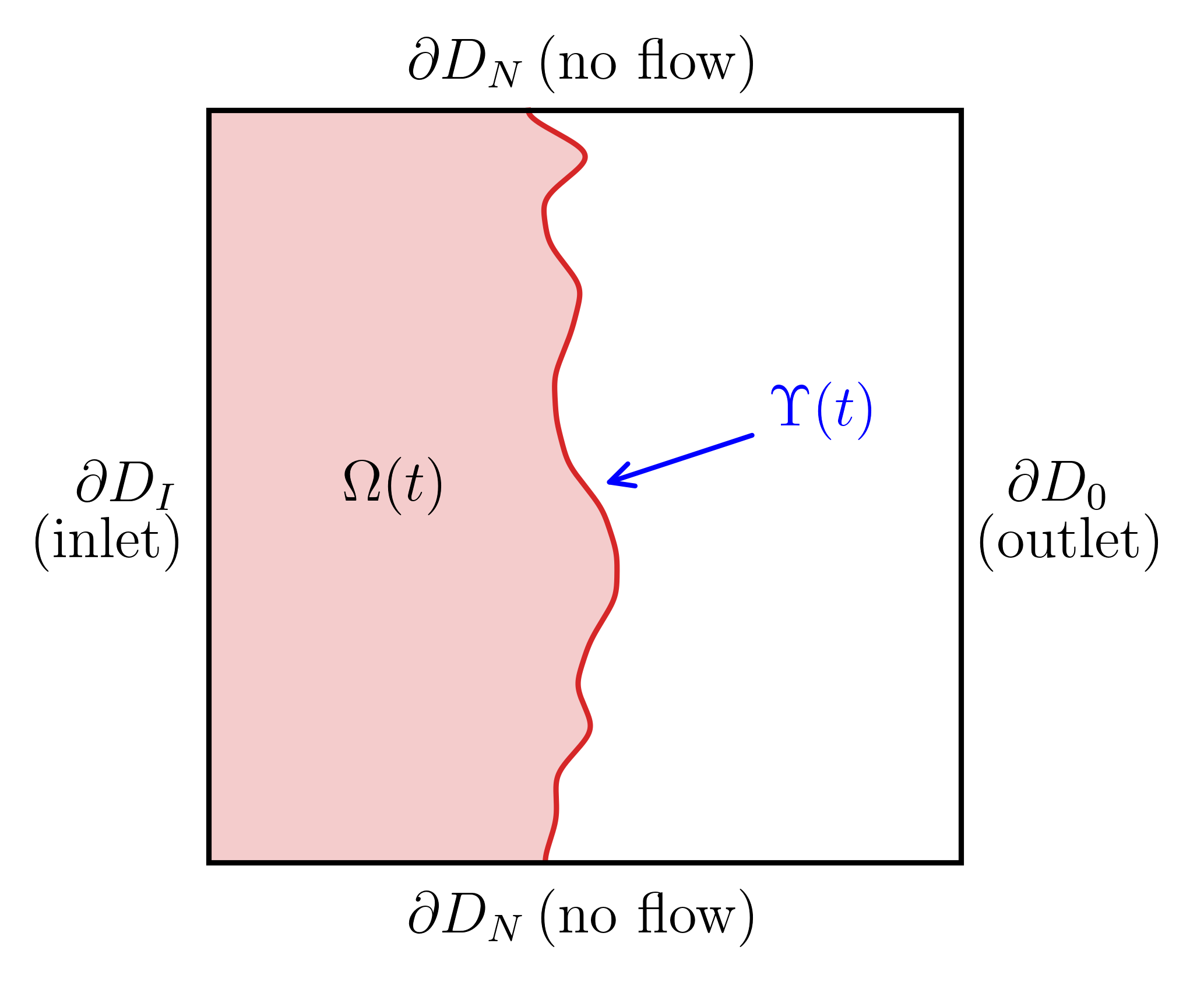}
    \caption{Example configuration.}
    \label{fig: setup_eg}
\end{figure}

The inverse problem considered in this work is that of inferring the
log-permeability field \(u(x)\) from pressure measurements collected from a
finite set of sensors distributed throughout the mould cavity. Such
pressure sensors are routinely embedded within RTM tooling and provide
time-dependent measurements of the evolving pressure field during the
infusion process \cite{Matveev,Causon2024}. 

\subsection{Industrial relevance}

Fibre-reinforced composites are used throughout the aerospace,
automotive, and marine industries \cite{Advani,Long}. Accurate
impregnation of the fibrous preform is crucial for the manufacture of
high-quality composite components. However, resin flow patterns are
strongly influenced by spatial variability in permeability, which is
inherently uncertain due to several factors including stochastic variability in
fibre arrangement, manufacturing imperfections during preform fabrication,
and localised deformation or wrinkling of the preform during layup
\cite{Andreas,Andreas2,Matveev2,Long2}. These effects induce random,
spatially varying deviations in permeability, leading to unpredictable
evolution of the moving resin front. Distorted flow patterns may produce
macroscopic voids or dry spots in the final component, significantly degrading
its mechanical properties and increasing manufacturing costs through wasted
material and rejected parts \cite{VARNA_voids_1995,Mehdikhani_voids,Advani}. 
The ability to infer permeability in real
time is therefore of considerable practical interest, particularly in the
context of defect detection and active process control. Consequently,
Bayesian inversion and uncertainty quantification for RTM processes have
received increasing attention in recent years (see, e.g. \cite{Causon2024,iglesias2025deeponet} and references therein), 
motivated both by advances in sensing technology and by the growing demand for reliable digital twins
and real-time monitoring strategies in composites manufacturing.

\subsection{Technical challenges}

Although Darcy flow in porous media has been extensively studied from
theoretical, numerical, and inverse perspectives, the moving
boundary formulation considered here differs fundamentally from the
classical fixed-domain Darcy models commonly used in groundwater and
reservoir simulation \cite{chen2006computational,bear1972dynamics}. In subsurface flow applications, invading and
displaced fluids typically coexist within partially saturated regions,
leading naturally to two-phase flow models involving coupled nonlinear
equations for pressure and saturation. Such models account for capillary
effects, relative permeability, and fluid mixing, and therefore do not
generally admit a sharp interface separating saturated and unsaturated
regions.

By contrast, the RTM model \eqref{eq:Darcy0}--\eqref{eq:Darcy2E} considered in this work adopts a
sharp-interface approximation in which the resin front is represented as a
moving boundary separating fully saturated and unsaturated regions. This
approximation is appropriate in many RTM settings due to the comparatively
small capillary diffusion and the relatively well-defined advancing resin
front observed experimentally. Similar sharp-interface formulations also
arise in certain subsurface applications under suitable assumptions,
including some models of \(\mathrm{CO}_2\) injection and seawater intrusion \cite{sharp,bear1972dynamics}.

Despite its practical relevance, the moving boundary model
\eqref{eq:Darcy0}--\eqref{eq:Darcy2E} introduces substantial mathematical
and computational challenges. The unknown permeability field \(u(x)\)
influences not only the pressure solution through the Darcy equation
\eqref{eq:Darcy0}, but also the evolution of the moving interface through
the kinematic condition \eqref{eq:Darcy1}. As a consequence, perturbations
in permeability induce coupled variations in both the state variable $p(x,t)$ and
the moving domain $\Omega(t)$. Deriving sensitivities of the forward model
therefore requires simultaneous linearisation of the pressure equation and
the moving boundary dynamics, substantially complicating adjoint
derivations, optimisation, and uncertainty quantification. To the best of
our knowledge, no variational method for the sharp-interface RTM
inverse problem has previously been developed, and corresponding
linearisations and adjoint systems suitable for gradient-based inversion are not currently available in the literature.

Although two-phase flow (fixed domain) formulations have also been proposed for modelling
resin infusion in RTM \cite{SHANTHAR2026109839}, such models are significantly more complicated
analytically and computationally, requiring the solution of strongly
coupled nonlinear systems together with more sophisticated and more computationally expensive numerical
discretisations. The sharp-interface formulation considered here therefore
provides an attractive compromise between physical realism and
computational tractability, particularly in the context of repeated
forward solves required for Bayesian inversion.

\subsection{Related work}

From the forward modelling perspective, explicit analytical solutions of
the moving boundary problem
\eqref{eq:Darcy0}--\eqref{eq:Darcy2E} are available only in one spatial
dimension. 
In higher dimensions, existence of a weak solution follows from \cite{BLZ2018} (see also references therein), but a corresponding uniqueness theory for spatially variable permeability on a general domain remains open.
Nevertheless, a range of numerical approaches for RTM
simulation have been developed and are routinely employed in both
open-source and industrial software
\cite{Advani,MichaelMinho,SHANTHAR2026109839}.

From the inverse perspective, existing approaches have largely
relied on derivative-free methodologies in which the forward solver is
treated as a black-box simulator, thereby avoiding the need for
linearisations and adjoint sensitivities. Among such methods,
ensemble Kalman inversion (EKI) has emerged as one of the principal
approaches for permeability inversion in RTM
\cite{Iglesias_2018,Matveev}, owing to its ability to approximate
posterior structure at substantially lower computational cost than
sampling-based methods such as Markov chain Monte Carlo (MCMC).
Nevertheless, repeated evaluations of the moving boundary problem remain
a major obstacle for real-time inversion and process control, since in
industrial settings a single high-fidelity simulation may require
several hours of computation.

Machine learning approaches for permeability estimation in RTM broadly
fall into two categories. The first seeks to learn the inverse map
directly from observations to permeability fields using supervised
learning techniques
\cite{Gonzalez,Gonzalez4,caglar2022deep,Hanna}. While such methods can
provide rapid predictions once trained, existing approaches have largely
been restricted to simplified or low-dimensional parameterisations in
which the permeability field is represented through a small number of
zones or geometric features. Moreover, the inverse map is typically
ill-posed, so that small perturbations in the observational data may
induce large changes in the inferred permeability. This raises
fundamental questions regarding robustness and generalisation beyond the
distribution represented in the training data.

The second class of approaches employs machine learning surrogates for the
forward model itself, including reduced-order models and neural operators,
within Bayesian inversion
\cite{Causon2024,iglesias2025deeponet,shuang}. Such methods can provide
substantial computational speedups once trained, but they generally
require large simulation datasets and introduce additional approximation
error. For example, the work of \cite{iglesias2025deeponet} reports
accurate inversion results using neural operators, but required
approximately \(10^4\) forward simulations for training. Furthermore, many neural operator architectures are
tied to fixed computational geometries, so changes in the mould
configuration may require retraining.

An additional challenge in Bayesian inversion is that surrogate-induced
errors must themselves be incorporated into the inference procedure. At
present, there is no broadly established approach for doing so
rigorously in the context of neural operators. Existing approaches, such
as enhanced model-error formulations \cite{OfflineUQ}, typically rely on
Gaussian assumptions for the surrogate error, despite the fact that such
approximation errors are generally neither Gaussian nor easily
characterised. Consequently, the impact of surrogate approximation on the
resulting posterior distribution remains difficult to quantify.

Together, these challenges motivate the development of variational
and adjoint-based methodologies capable of exploiting the structure of the
underlying PDE-constrained inverse problem directly, thereby leading to
efficient gradient-based optimisation and uncertainty quantification
without requiring large ensembles, extensive offline training, or repeated
black-box forward simulations.

\subsection{Contributions}

This work develops an infinite-dimensional Bayesian framework for moving
boundary inverse problems arising in resin transfer moulding. In
particular, we derive a variational formulation of the resin infusion
model \eqref{eq:Darcy0}--\eqref{eq:Darcy2E}, together with the associated
linearised and adjoint systems required for efficient gradient-based
inference. To the best of our knowledge, this is the first work to
combine shape-differentiable linearisations, adjoint-based optimisation,
and reduced posterior approximations within a unified Bayesian framework
for this class of Darcy-flow moving boundary problems.

The inverse problem is formulated using Gaussian Matérn priors on the
log-permeability field. Maximum a posteriori (MAP) estimators are
computed through minimisation of the Onsager--Machlup functional using a
reduced Levenberg--Marquardt scheme posed on the Cameron--Martin space of
the prior. Central to the methodology is the derivation of a representer
structure for the Fréchet derivatives of the observation operator. This
allows both the reduced gradient and the Gauss--Newton Hessian to be
expressed through low-dimensional matrices in observation space, while
the dependence on the infinite-dimensional parameter field is encoded
through a finite collection of representer functions associated with the
observations.

The same reduced structure is then used to construct a linearised MAP
(LMAP) approximation of the posterior distribution. In particular, the
posterior covariance is approximated using the Gauss--Newton Hessian
evaluated at the MAP estimator, yielding an explicit finite-rank
representation of the approximate posterior covariance in terms of the
representers. Consequently, both MAP estimation and posterior
approximation reduce to repeated forward and adjoint solves together with
low-dimensional linear algebra operations in observation space.

The characterisation of the representers constitutes the principal
analytical component of the work. In one spatial dimension (1D), where
explicit solutions of the moving boundary problem are available, we
derive closed-form expressions for the Fréchet derivatives and the
associated representers. This yields explicit formulas for the reduced
gradient and posterior covariance approximation and provides a rigorous
consistency check for the Bayesian framework.

In two spatial dimensions (2D), explicit analytical solutions are unavailable.
To address this, we introduce a very weak mixed formulation in which the
geometry and flow dynamics are coupled weakly through the interface
conditions. This formulation enables formal shape differentiation of the
forward map with respect to the log-permeability field and leads
naturally to the derivation of the corresponding linearised and adjoint
systems. Beyond the specific optimisation method considered here, the
resulting representer structure is sufficiently general to be integrated
into a broad class of derivative-informed deterministic and Bayesian
inversion methodologies for moving boundary problems.

From the computational perspective, the proposed methodology is
implemented through repeated forward and adjoint solves within the
reduced Levenberg--Marquardt method. In 1D, the
numerical experiments additionally validate the proposed Bayesian
approximation framework by benchmarking the LMAP approximation against
both MCMC and EKI. The results demonstrate that the Gaussian
approximation centred at the MAP captures the posterior distribution
accurately in the considered settings while being substantially cheaper
computationally.

In 2D, where repeated MCMC sampling becomes
computationally prohibitive, the numerical investigations focus on
comparisons between the proposed LMAP methodology and EKI. The results
demonstrate that the adjoint-based reduced framework provides
substantial computational advantages while retaining comparable
reconstruction quality and uncertainty quantification. The
methodology is applied to complex mould geometries, illustrating both
the flexibility of the variational formulation and the ability of the
framework to handle realistic moving boundary configurations beyond
simplified benchmark domains.

While the use of adjoint methods requires access to and modification of
the forward solver, the numerical experiments demonstrate that this
additional implementation effort yields substantial computational gains,
enabling accurate inversion with significantly fewer forward solves than
ensemble-based or derivative-free methodologies.

The remainder of the paper is organised as follows.
In Section~\ref{sec:problem}, we introduce the Bayesian inverse problem
setting together with the MAP estimators, the reduced
Levenberg--Marquardt formulation, and the associated LMAP posterior
approximation. Section~\ref{sec:representers} is devoted to the
derivation of the representers and the corresponding linearised and
adjoint systems. In Section~\ref{sec:numerics}, we discuss the numerical
implementation and present a range of one- and two-dimensional numerical
experiments. Finally, conclusions and directions for future work are
presented in Section~\ref{sec: Conclusion}.
Proofs and additional computational details are given in Appendices.

\section{Variational framework for Bayesian inversion}
\label{sec:problem}

\subsection{The forward map}

The forward model is given by
\eqref{eq:Darcy0}--\eqref{eq:Darcy2E}. We first note that the pressure
field \(p(x,t)\) is only physically defined within the saturated region
\(\Omega(t)\). For variational and inverse problem
purposes, it is convenient to regard \(p\) as an extension to the
hold-all domain \(D\). Throughout this work, the porosity \(\phi\) is
assumed constant for simplicity.

An important quantity in the subsequent analysis is the time at which the
entire domain becomes saturated (see \cite{MichaelMinho}):
\begin{equation}\label{eq:filltime}
\tau^*
:=
\inf\{t>0\,|\, \overline{\Omega}(t)=\overline D\},    
\end{equation}
referred to as the \emph{filling time}. For \(t\geq\tau^*\), the moving
boundary disappears and the system \eqref{eq:Darcy0}--\eqref{eq:Darcy2E} reduces to the standard steady-state
Darcy problem posed on the fixed domain \(D\).

The quantities \(p\), \(\Upsilon\), and \(\tau^*\) depend on
the log-permeability field \(u\). Where convenient, this dependence is
made explicit through the notation
\(
p[u]
\),
\(
\Upsilon[u]
\),
and
\(
\tau^*[u]
\).
Although the moving boundary is itself determined by \(u\), the inverse
problem considered here is formulated solely in terms of pressure
observations. Accordingly, we consider only the pressure component of the
forward operator $u\mapsto p[u]$.

Our objective is to infer the log-permeability field \(u(x)\) from
finitely many noisy observations of the pressure field over the
space--time domain \(D_T:=D\times(0,T]\), where \(T>0\) denotes the simulation/observation horizon. Motivated both by the analytical structure of the
1D solution discussed in Section \ref{sec:1D} and by the variational formulation introduced
in Section \ref{sec:2D} for the 2D setting, we work throughout with the
functional setting
\[
u\in X:=C(\overline D),
\qquad
p[u]\in
L^2(0,T;H^1(D))
\cap L^\infty(D_T).
\]

To remain consistent with this framework, observations are modelled as
bounded linear functionals of the pressure field. Specifically, we
consider a collection of observation functionals
\(
\{\mathcal G_i\}_{i=1}^M
\)
of the form
\begin{equation}
\mathcal G_i(u)
=
\int_0^T\int_D
\mathcal H_i(x,t)\,p[u](x,t)\,dx\,dt,
\qquad i=1,\ldots,M,
\label{eq:observation_functionals}
\end{equation}
for kernels
\(
\mathcal H_i\in L^\infty(D_T)
\).
This assumption ensures that each observation functional is well defined
and continuous with respect to the pressure field, while remaining
sufficiently general to model localised measurements.

The forward (parameter-to-observable) map
\(
\mathcal G:X\to\mathbb R^M
\)
is therefore defined by
\begin{equation}
\mathcal G(u)
=
\big(
\mathcal G_1(u),\ldots,\mathcal G_M(u)
\big)
\in\mathbb R^M,
\label{eq: Forward_Map}
\end{equation}
where \(X=C(\overline D)\) is equipped with the supremum norm.

Observations are assumed to be corrupted by additive Gaussian noise:
\begin{equation}\label{eq:obsmodel}
y = \mathcal G(u) + \eta,
\end{equation}
where \(\eta \sim \mathcal N(0,\Sigma)\) with known covariance
\(\Sigma \in \mathbb R^{M \times M}\). The inverse problem is therefore to recover \(u\) from a realisation of the data \(y\) given
the nonlinear forward map \(\mathcal G\).

In applications, each functional \(\mathcal G_i\) is associated with a
sensor located at a space--time point \((x_i,t_i)\), and the corresponding
kernel \(\mathcal H_i\) is chosen to be localised in a neighbourhood of
this point. In this sense, \(\mathcal G_i(u)\) may be interpreted as a
regularised approximation of the pointwise observation \(p[u](x_i,t_i)\),
with the regularisation reflecting both analytical requirements and the
finite resolution of numerical discretisations.

\subsection{Bayesian formulation}\label{sec:Bayes}

A Gaussian prior \(\mu_0 = \mathcal{N}(\bar u, \mathcal{C}_0)\) is placed on \(u\), where the covariance operator \(\mathcal{C}_0\) is a self-adjoint, positive, trace-class operator on \(L^2(D)\). Attention is restricted to Matérn-class covariance operators, represented in operator form by 
\[
\mathcal C_0
=
\sigma^2(\varkappa^2 I-\Delta)^{-(\nu+d/2)},
\]
where \(\sigma^2\) controls the marginal variance while \(\varkappa\) determines the intrinsic lengthscale $l=\sqrt{8\nu}/\varkappa$ \cite{LINDGREN2022100599}. The fractional power of the Laplacian is understood spectrally together with appropriate boundary conditions on \(D\). Under standard assumptions on \(D\), samples from \(\mu_0\) belong
almost surely to \(H^s(D)\) for every \(s<\nu\). We therefore assume $\nu>d/2$, so that the prior has continuous sample paths and
\(
\mu_0\bigl(C(\overline D)\bigr)=1
\). Thus, with
\(
X:=C(\overline D)
\), the prior may be regarded as a Gaussian measure on \(X\).

The Cameron--Martin (CM) space \(E\) associated with \(\mu_0\) is
\[
E=\operatorname{Range}(\mathcal C_0^{1/2}),
\]
equipped with the inner product
\[
\langle h_1,h_2\rangle_E
=
\left\langle
\mathcal C_0^{-1/2}h_1,
\mathcal C_0^{-1/2}h_2
\right\rangle_{L^2(D)}.
\]
For covariance operators of the form above, and subject to the boundary
conditions used in the spectral definition of the operator, the
CM space \(E\) is norm-equivalent to
\(H^{\nu+d/2}(D)\) \cite{BolinKirchner2023}. Consequently, \(E\) is compactly embedded in
\(C(\overline D)\) whenever \(\nu>0\).

Given the observational model \eqref{eq:obsmodel}, the likelihood is given by
\begin{align}
\pi(y \mid u) \propto \exp\big(-\Phi_{y}(u)\big), 
\qquad 
\Phi_{y}(u) := \frac{1}{2}\|y - \mathcal{G}(u)\|_{\Sigma}^2,
\label{eq:likelihood}
\end{align}
where
\(
\|v\|_\Sigma^2
:=
v^\top\Sigma^{-1}v,
\) and \(\Sigma\) is assumed symmetric and positive definite.

Under suitable assumptions on \(\GG\), the posterior measure \(\mu^y\) is absolutely continuous with respect to the prior measure \(\mu_0\), with Radon--Nikodym derivative \cite{Stuart}:
\begin{equation}\label{eq:RN}
\frac{d\mu^y}{d\mu_0}(u)
\propto
\exp\big(-\Phi_y(u)\big).
\end{equation}
MAP estimators, interpreted in the small-ball probability sense \cite{dashti2013map}, are characterised as minimisers of the Onsager--Machlup functional
\begin{align} \label{eq: Onsager_Machlup}
    J_{\mathrm{OM}}(u)
    :=
    \begin{cases}
        \Phi_y(u)
        +
        \dfrac12\lVert u-\bar u\rVert_E^2,
        & u-\bar u\in E,\\
        +\infty,
        & u-\bar u\in X\setminus E.
    \end{cases}
\end{align}
We assume throughout that \(\bar u\in E\). It then follows directly from
\eqref{eq: Onsager_Machlup} that any MAP estimator \(u_{\mathrm{MAP}}\) with finite Onsager--Machlup functional belongs to \(E\).

In Section~\ref{sec:1D} we verify the above assumptions explicitly for the
1D version of \eqref{eq:Darcy0}--\eqref{eq:Darcy2E} and establish existence of
minimisers of \(J_{\mathrm{OM}}\). In this setting, the forward map \(\GG\)
admits a closed-form representation, enabling a rigorous derivation of its
Fréchet derivative \(D\GG\). In the 2D case, explicit analytical expressions for the forward
map are not available. 
As noted in the Introduction, existence of a weak forward solution is known \cite{BLZ2018}, but a corresponding uniqueness and 
differentiability theory is not available at this level of generality, and its establishment is beyond the scope of the present work.
We therefore proceed
under analogous regularity assumptions on \(\GG\), and derive the required
Fréchet derivatives from a weak formulation using shape calculus and
adjoint-based techniques.

The derivation of Fréchet derivatives and the associated adjoint-based
representer structure is deferred to
Section~\ref{sec:representers}. For the moment, we proceed under the assumption that suitable derivative
representations are available and focus instead on the resulting reduced
optimisation framework for MAP estimation through minimisation of the
Onsager--Machlup functional \eqref{eq: Onsager_Machlup}. In addition, we consider the standard linearised MAP (LMAP)
approximation \cite{Ghattas1,Ghattas2}, which provides a practical
Gaussian approximation of the posterior measure of the form
\begin{equation}
\mu^y
\approx
\mathcal N(u_{\mathrm{MAP}},\mathcal C_{\mathrm{MAP}}).
\label{eq:Gapp}
\end{equation}
The covariance operator \(\mathcal C_{\mathrm{MAP}}\) is defined through
the Gauss--Newton approximation of the posterior Hessian:
\begin{align}\label{eq:Cov}
\mathcal{C}_{\mathrm{MAP}}
:=
\mathcal C_0
-
\mathcal C_0D\GG(u_{\mathrm{MAP}})^\ast
\left(
D\GG(u_{\mathrm{MAP}})
\mathcal C_0
D\GG(u_{\mathrm{MAP}})^\ast
+\Sigma
\right)^{-1}
D\GG(u_{\mathrm{MAP}})
\mathcal C_0,
\end{align}
where \(D\GG(u_{\mathrm{MAP}})^\ast\) denotes the adjoint of
\(D\GG(u_{\mathrm{MAP}})\), which, as shown in the following sections, admits
an extension to \(L^2(D)\).

The remainder of this section is devoted to the efficient computation of local
minimisers of \eqref{eq: Onsager_Machlup}.

\subsection{Levenberg--Marquardt scheme}\label{sec:LM}

A classical approach for computing MAP estimators in PDE-constrained
Bayesian inversion is to formulate the optimisation problem as the
minimisation of the Onsager--Machlup functional subject to the governing
PDE constraints, and then derive the associated
Karush--Kuhn--Tucker optimality system involving the state, parameter,
and adjoint variables \cite{Ghattas1,doi:10.1137/130934805}. This full-space
formulation leads to a nonlinear saddle-point system, typically solved
through Newton-type linearisation together with preconditioned Krylov
subspace methods. While effective, such approaches require the
solution of large-scale systems whose dimension scales with both the state and
parameter discretisations.

In this work, we adopt a reduced formulation in which the PDE constraint is
incorporated through the forward map \(\mathcal G\), and the optimisation is carried out directly over the CM space \(E\).
Let \(u_k - \bar u \in E\) denote the (mean-adjusted) current iterate, and assume that
\(\mathcal G\) is Fréchet differentiable at \(u_k\). Using the first-order
approximation $\mathcal G(u_k + h)
\approx
\mathcal G(u_k) + D\mathcal G(u_k)h$, we define the Levenberg--Marquardt (LM) iteration \(u_{k+1} = u_k + h_k\),
where the increment \(h_k \in E\) is the minimiser of
\begin{align}
J_k(h)
&=
\frac12
\|y - \mathcal G(u_k) - D\mathcal G(u_k)h\|_\Sigma^2
+
\frac12
\|u_k + h - \bar u\|_E^2
+
\frac{\alpha_k}{2}\|h\|_E^2,
\qquad h \in E.
\label{eq:LM_cost}
\end{align}
The functional \( J_k\) is continuous, coercive, and strictly convex on
\(E\), and therefore admits a unique minimiser \(h_k \in E\). The case \(\alpha_k = 0\) corresponds to the Gauss--Newton method, while
\(\alpha_k > 0\) yields the Levenberg--Marquardt regularisation, which improves
stability of the linearised problem.

\subsection{Representer formulation and reduced structure}\label{sec:reduced}

We now exploit the finite-dimensional structure of the data. For a given
iterate \(u_k \in E\), the analysis from Section~\ref{sec:representers} shows that for each component $\GG_{i}$ of the forward map defined in \eqref{eq:observation_functionals} there exists a \textit{representer} function \(r_i[u_k] \in L^2(D)\),
\(i=1,\ldots,M\), such that
\begin{equation}
D\GG_{i}(u_{k}) h
=
\langle r_i[u_k], h \rangle_{L^2(D)},
\qquad \forall \ h \in E.
\label{eq:L2_rep}
\end{equation}
Hence, the Fréchet derivative of the forward map can be written as
\begin{equation}
D\GG(u_{k}) h=\Big(\langle r_1[u_k], h \rangle_{L^{2}(D)},\dots, \langle r_M[u_k], h \rangle_{L^{2}(D)}\Big) .
\label{eq:DG_rep}
\end{equation}
Because each \(r_i[u_k]\in L^2(D)\), formula~\eqref{eq:DG_rep}
defines a bounded extension of
\(
D\mathcal G(u_k):E\longrightarrow\mathbb R^M
\)
to an operator on \(L^2(D)\). For the Matérn covariance operators considered here, the CM space \(E\) is dense in \(L^2(D)\) \cite{Dashti2017}. Hence, this
bounded extension is unique. Its adjoint
\(
D\mathcal G(u_k)^*:\mathbb R^M\longrightarrow L^2(D)
\)
is therefore given by
\begin{equation}
D\GG(u_{k})^{\ast} w = \sum_{i=1}^M w_i r_i[u_k].
\label{eq:DG_rep_star}
\end{equation}

The interaction between the observation directions is encoded in the matrix
\begin{equation}
[\mathcal R(u_k)]_{ij}
:= [D\GG(u_{k})\mathcal C_0 D\GG(u_{k})^{\ast}]_{ij}
=
\langle r_i[u_k], \mathcal C_0r_j[u_k] \rangle_{L^{2}(D)}.
\label{eq:mat_rep}
\end{equation}

The following theorem (see Appendix~\ref{app: lm_descent} for its proof) provides a reduced formulation of both the LM update and the LMAP covariance.

\begin{theorem}[Reduced LM update and LMAP covariance]
\label{thm:LM_LMAP_clean}
Let \(u_k \in E\) and suppose that \(\mathcal G\) is Fréchet differentiable at
\(u_k\) and that there exist $r_{i}[u_{k}]\in L^{2}(D)$, $i=1,\dots, M$, satisfying \eqref{eq:L2_rep}. Then \\
\noindent
(i) the LM update \(h_k \in E\) is given by
\begin{equation}
h_k = \frac{\bar u - u_k}{1 + \alpha_k}+\sum_{i=1}^M a_{k,i}\, \mathcal C_0r_i[u_{k}],
\label{eq:MAP_update}
\end{equation}
where \(a_k=(a_{k,1},\dots,a_{k,M})\) solve
\begin{equation}
\bigl[\mathcal R(u_k) + (1 + \alpha_k)\Sigma\bigr] a_k
= y - \mathcal G(u_k)-
D\GG(u_{k})\!\left(\frac{\bar u - u_k}{1 + \alpha_k}\right);
\label{eq:MAP_update2}
\end{equation}

\medskip

\noindent
(ii) for \(u_{\mathrm{MAP}}\) being a stationary point of \(J_{\mathrm{OM}}\), the Gauss--Newton Hessian covariance in \eqref{eq:Cov} admits the finite-rank representation
\begin{equation}
\label{eq:CMAP_rep}
\mathcal C_{\mathrm{MAP}} v
=\mathcal C_0 v - \mathcal C_0\sum_{i,j=1}^M
r_i[u_{\mathrm{MAP}}]\,
\bigl[\mathcal R(u_{\mathrm{MAP}}) + \Sigma\bigr]^{-1}_{ij}\,
\langle r_j[u_{\mathrm{MAP}}], \mathcal C_0 v \rangle_{L^2(D)},
\qquad v \in L^2(D),
\end{equation}
and hence differs from \(\mathcal C_0\) by an operator of rank at most \(M\).
\end{theorem}

Provided \(\bar u,u_0\in E\), repeated application of
Theorem~\ref{thm:LM_LMAP_clean} yields a LM iterative
scheme in function space generating a sequence of iterates \(\{u_k\}_{k=0}^K \subset E\).

A notable feature of the reduced formulation in
Theorem~\ref{thm:LM_LMAP_clean} is that both the optimisation step and the
local Gaussian approximation are entirely determined by the representers
\(r_i[u_k]\). Once these functions have been computed, the LM update requires
only the assembly of the reduced \(M\times M\) matrix
\(\mathcal R(u_k)\) and the solution of the finite-dimensional linear system
\eqref{eq:MAP_update2}. The resulting update \(h_k\) is then obtained
explicitly from \eqref{eq:MAP_update}.

Consequently, the dominant computational cost per iteration lies in the evaluation of the
forward model and the construction of the representers, typically through
adjoint solves. All subsequent computations are performed in the reduced data
space \(\mathbb R^M\), independently of the discretisation dimension of the
parameter and state variables. Moreover, the same reduced matrix
\(\mathcal R(u_k)\) governs both the LM update and the LMAP covariance
approximation. This reduced formulation is particularly advantageous in the
practically relevant regime where the number of observations \(M\) is small
compared with the number of parameter and state degrees of freedom. In this
sense, the reduced system may be interpreted as the Schur complement of the
linearised PDE-constrained optimality system (see, e.g. \cite{SchurRef}).

The damping parameter $\alpha_k$ is updated using a simple multiplicative
accept--reject strategy based on the standard LM principle: it is decreased following an accepted step and increased
following a rejected step \cite{LM_original}. Successful iterations reduce the damping,
thereby recovering Gauss--Newton behaviour near a minimiser, while rejected
iterations increase the damping to improve robustness of the linearised
subproblem. The iterations are terminated when the relative change in the
Onsager--Machlup functional falls below a prescribed tolerance. We summarise the resulting procedure in
Algorithm~\ref{alg:LM}. Upon convergence, the
final representers and reduced matrix
\(\mathcal R(u_{\mathrm{MAP}})\) are used to construct the LMAP covariance
operator \(\mathcal C_{\mathrm{MAP}}\), thereby yielding the Gaussian
approximation introduced in \eqref{eq:Gapp} used for uncertainty quantification. 

The remainder of the paper is devoted to the derivation and efficient
computation of the representers for both one- and two-dimensional moving
boundary problems, together with numerical experiments illustrating the
resulting MAP and the corresponding LMAP approximation of the posterior.

\begin{algorithm}[h]
\caption{Reduced Levenberg--Marquardt algorithm}
\label{alg:LM}
\begin{algorithmic}[1]
\Require Initial damping \(\alpha_0>0\), damping factor
\(\gamma\in(0,1)\), tolerance \(\varepsilon_J>0\), maximum number of iterations \(N_{\max}\),
prior mean \(\bar u\in E\)
\State Initialise \(k\gets 0\), \(u\gets \bar u\), \(\alpha\gets \alpha_0\)
\State \textbf{Evaluate forward objective at \(u\):}
\Statex \hspace{\algorithmicindent} (i) Solve \eqref{eq:Darcy0}--\eqref{eq:Darcy2E} to obtain \(p[u]\), \(\Upsilon[u]\), and \(\tau^*[u]\)
\Statex \hspace{\algorithmicindent} (ii) Evaluate the forward map \(\mathcal G(u)\) from \eqref{eq: Forward_Map}
\Statex \hspace{\algorithmicindent} (iii) Compute \(J_{\mathrm{OM}}\gets J_{\mathrm{OM}}(u)\)
    
\While{\(k<N_{\max}\)}
    \State Compute representers \(r_i[u]\), \(i=1,\ldots,M\)
    \State Assemble the reduced matrix \(\mathcal R(u)\) from \eqref{eq:mat_rep}
    \State Compute \(h\) from \eqref{eq:MAP_update}--\eqref{eq:MAP_update2}
    \State Set \(\widehat u\gets u+h\)
    \State \textbf{Evaluate forward objective at \(\widehat u\):}
    \Statex \hspace{\algorithmicindent} (i) Solve \eqref{eq:Darcy0}--\eqref{eq:Darcy2E} to obtain \(p[\widehat u]\), \(\Upsilon[\widehat u]\), and \(\tau^*[\widehat u]\)
    \Statex \hspace{\algorithmicindent} (ii) Evaluate the forward map \(\mathcal G(\widehat u)\) from \eqref{eq: Forward_Map}
    \Statex \hspace{\algorithmicindent} (iii) Compute \(\widehat J_{\mathrm{OM}}\gets J_{\mathrm{OM}}(\widehat u)\)

    \If{\(\widehat J_{\mathrm{OM}}<J_{\mathrm{OM}}\)}
        \If{\(
  \dfrac{
\left|\widehat J_{\mathrm{OM}}-J_{\mathrm{OM}}\right|
}{
\left|J_{\mathrm{OM}}\right|
}
\leq\varepsilon_J
\)
        }
            \State \Return \(\widehat u\)
        \EndIf
        \State \(u\gets \widehat u\)
        \State \(J_{\mathrm{OM}}\gets \widehat J_{\mathrm{OM}}\)
        \State \(\alpha\gets \gamma\alpha\)
    \Else
        \State \(\alpha\gets \alpha/\gamma\)
    \EndIf

    \State \(k\gets k+1\)
\EndWhile
\State \Return \(u\)
\end{algorithmic}
\end{algorithm}

\section{Derivation of representers}\label{sec:representers}

\subsection{One-dimensional case} \label{sec:1D}

In this section we analyse the inverse problem in 1D. In
this setting, the moving boundary problem admits closed-form solutions,
allowing rigorous analysis of the associated forward map. In particular, we
establish well-posedness of the Bayesian inverse problem, existence of MAP
estimators, and explicit representer formulas for the Fréchet derivative of
the observation operator.

\subsubsection{Well-posedness of the Bayesian inverse problem and MAP estimators.}
In 1D, the porous medium occupies the physical domain $D:=(0,x^*)$, with $x^* > 0$. The moving boundary problem \eqref{eq:Darcy0}-\eqref{eq:Darcy2E} reduces to:
\begin{align}
    -\frac{d}{dx}\Big[e^{u(x)}\frac{dp}{dx}(x,t) \Big] &= 0,\ &x \in (0,\Upsilon(t)),\ t>0,\label{eq: 1D_start}\\
    \frac{d\Upsilon}{dt}(t) + \frac{1}{\mu_f\phi}e^{u(\Upsilon(t))}\frac{dp}{dx}(\Upsilon(t),t) &= 0,\ &\Upsilon(0) = 0,\ t \geq 0, \label{eq: 1DB}\\
    p(0,t) &= p_I,\ &t\geq 0, \notag \\
    p(\Upsilon(t),t) &= p_0,\ &t > 0, \notag \\
    p(x,0) &= p_0,\ &x \in (0,x^*] . \label{eq: 1D_end}
\end{align}
If the specified simulation time $T > 0$ is greater than the filling time $\tau^*$ (see eq. \eqref{eq:filltime}), the condition $\Upsilon(t) = x^*$ is applied for each $t \geq \tau^*$. 

For \(u\in X=C(\overline D)\) and \(t\in(0,T]\), the moving boundary problem \eqref{eq: 1D_start}-\eqref{eq: 1D_end} admits closed form solution given by \cite{Advani,MichaelMinho,Tartakovsky}: 
\begin{align}\label{eq: Analytical1D_start}
        \Upsilon[u](t) &= \begin{cases}
            W^{-1}[u]\big(\frac{p_I-p_0}{\mu_f\phi}t\big), &t < \tau^*[u],\\
            x^*, &t \geq \tau^*[u],
        \end{cases}\\
        p[u](x,t) &= \begin{cases} 
        p_I - (p_I - p_0)\frac{F[u](x)}{F[u](\Upsilon[u](t))}, &x \in \Omega[u](t) :=(0,\Upsilon[u](t)),\\
        p_0, &x \in D\backslash \Omega[u](t),
        \end{cases}\label{eq: Analytical1D_end}
    \end{align}
where 
    \begin{align} \label{eq: FandW}
        F[u](x) := \int_0^x e^{-u(z)}\,dz,\ \ \ W[u](x) := \int_0^x F[u](\xi)\,d\xi,
    \end{align}
and the filling time is given by   
\begin{equation} \label{eq:fill1D}
\tau^*[u]=\frac{\mu_{f}\phi}{p_{I}-p_{0}}W[u](x^*).
\end{equation}

The explicit representation \eqref{eq: Analytical1D_start}--\eqref{eq: Analytical1D_end} allows us to establish local regularity of the
forward map and hence well-posedness of the corresponding Bayesian inverse
problem.

\begin{theorem}[Posterior well-posedness and MAP estimators] \label{the:Bayesian1D}
Consider the 1D forward model described by \eqref{eq: 1D_start}--\eqref{eq: 1D_end}, together with the Matérn Gaussian
prior introduced in Section~\ref{sec:Bayes}. Assume that the observation operator
\(\GG:X\to\mathbb R^M\) is defined by \eqref{eq: Forward_Map}. Then the conditional distribution \(u\vert y\) is given by the posterior
measure \(\mu^y\), which is absolutely continuous with respect to \(\mu_0\)
and satisfies \eqref{eq:RN}. Moreover, \(\mu^y\) admits at least one MAP
estimator in the sense of small-ball probabilities, and every such MAP
estimator is characterised as a minimiser of the Onsager--Machlup functional
\(J_{\mathrm{OM}}\) defined in \eqref{eq: Onsager_Machlup}.
 \begin{proof}
By Lemma~A.2 of \cite{Iglesias_2018}, the pressure field depends locally
Lipschitz continuously on the log-permeability field \(u\). Since the
observation kernels \(\mathcal H_i\) belong to \(L^\infty(D_T)\), it follows
directly from \eqref{eq:observation_functionals} that the observation operator
\(\GG:X\to\mathbb R^M\) is locally Lipschitz continuous. Consequently, the data-misfit functional $\Phi_y(u)$ defined in \eqref{eq:likelihood} satisfies Assumption~2.1 of \cite{dashti2013map}. Standard results from Bayesian inverse problems then
imply that the posterior measure \(\mu^y\) is well defined \cite{Stuart}, while
Theorem~3.5 and Corollary 3.10 of \cite{dashti2013map} yield existence and characterisation of MAP estimators, respectively.
\end{proof}
\end{theorem}

\subsubsection{Fréchet derivative and representer-based characterisation.} 
We next derive an explicit representation of the Fréchet derivative of the observation operator. This result forms the basis for the reduced
LM algorithm for MAP estimation and LMAP approximation introduced in Section \ref{sec:reduced}.

\begin{theorem}[Fréchet differentiability and representers]\label{the:G_differentiable}
Let $\mathcal G:X\to \mathbb R^M$ be the forward map whose components are
defined by \eqref{eq:observation_functionals} with $\mathcal H_i\in L^\infty(D_{T})$ for $i=1,\ldots,M$. Then $\mathcal G_i$ is Fréchet differentiable over the CM space $E$. Moreover, for every $h\in E$,
\[
    D\mathcal G_i(u)h
    =
    \langle r_i[u],h\rangle_{L^2(D)},
\]
where
\begin{align}
r_i[u](z)
&=
(p_I-p_0)e^{-u(z)}
\int_0^T
\mathbb{I}_{\{z<\Upsilon(t)\}}
\Bigg[
\frac{1}{F[u](\Upsilon(t))}
\int_z^{\Upsilon(t)} \mathcal H_i(x,t)\,dx
\nonumber\\
&\quad
-
\frac{1}{F[u](\Upsilon(t))^2}
\int_0^{\Upsilon(t)} \mathcal H_i(x,t)F[u](x)\,dx
\nonumber\\
&\quad
+
\mathbb{I}_{\{t<\tau^*[u]\}}
\frac{\Upsilon(t)-z}{F[u](\Upsilon(t))^3} 
e^{-u(\Upsilon(t))}
\int_0^{\Upsilon(t)} \mathcal H_i(x,t)F[u](x)\,dx
\Bigg]\,dt.
\label{eq:ri_simplified}
\end{align}
\end{theorem}
    \begin{proof}
        See Appendix~\ref{app:G_differentiable}.
    \end{proof}

\subsubsection{Numerical implementation.}
\label{sec:1Dimple}

All 1D numerical experiments presented later in this work were
implemented in MATLAB. The explicit analytical representation
\eqref{eq: Analytical1D_start}--\eqref{eq: Analytical1D_end} enables
efficient evaluation of the forward objective required in Steps~2 and~8
of Algorithm~\ref{alg:LM}.

The space--time domain \(D_T\) is discretised using a uniform grid. For a
given log-permeability field \(u\), the moving interface
\(\Upsilon(t)\) is computed numerically by solving the moving front
equation \eqref{eq: 1DB} using an implicit time-stepping scheme on the
temporal mesh; see Appendix~C of \cite{Iglesias_2018}. Once
\(\Upsilon(t)\) has been obtained, the pressure field \(p[u]\) is
evaluated directly from the analytical expression
\eqref{eq: Analytical1D_end} over the resulting space--time grid. The
corresponding filling time \(\tau^*[u]\) is computed from the explicit
formula following \eqref{eq:fill1D}.

The observation kernels \(\mathcal H_i\) are associated with prescribed
sensor locations and observation times
\(
(x_i,t_i)
\).
Numerically, the observations are approximated through local cell
averages of the pressure field over the cells of the space--time grid
containing the corresponding sensor locations. The Onsager--Machlup
functional \(J_{\mathrm{OM}}\) is then evaluated using midpoint
quadrature over the resulting discretisation. For the computation of the representers, the numerical approximations of
\(p[u]\), \(\Upsilon[u]\), and \(\tau^*[u]\) are substituted directly
into the explicit representation formula
\eqref{eq:ri_simplified}. The reduced LM update
\eqref{eq:MAP_update}--\eqref{eq:MAP_update2}, together with the LMAP
covariance approximation \eqref{eq:CMAP_rep}, are subsequently assembled
using midpoint quadrature.

In Section~\ref{sec:LMAP1D}, we compute numerical MAP estimators together
with their associated LMAP posterior approximations for representative
1D filling problems.


\subsection{Two-dimensional analysis}
\label{sec:2D}

In contrast to the 1D setting considered in
Section~\ref{sec:1D}, explicit analytical representations of the moving
boundary problem
\eqref{eq:Darcy0}--\eqref{eq:Darcy2E} are not available in 2D. To the best of our knowledge, a complete
well-posedness theory for the coupled Darcy moving boundary problem also
remains open. In practice, numerical solvers are typically constructed under the
assumption that, for each fixed time \(t>0\), the pressure field satisfies
a standard elliptic Darcy problem on the current saturated region
\(\Omega(t)\), namely
\(
p(\cdot,t)\in H^1(\Omega(t))
\),
together with the boundary conditions 
\eqref{eq:Darcy2E}, including the condition
\(p=p_0\) on the moving boundary \(\Upsilon(t)\). This yields a standard
finite element problem for the pressure field. Once \(p(\cdot,t)\) has
been computed, the kinematic condition \eqref{eq:Darcy1} is then used to
update the moving front through a control-volume-type transport step; see
for example \cite{MichaelMinho}. While such approaches have proven highly
successful in practice, rigorous convergence theory remains limited.

From the inverse problem perspective, an additional difficulty arises from
the fact that the pressure solution is defined on a moving domain that
depends implicitly on the permeability field itself. Consequently,
perturbations of the parameter field induce coupled perturbations of both
pressure solution and geometry of the evolving interface. This
substantially complicates sensitivity analysis, since standard
differentiation arguments do not apply directly on evolving domains. 
The objective of this section is therefore to develop a variational
framework suitable for shape differentiation, adjoint-based sensitivity
analysis, and the efficient computation of representers for the reduced
LM and LMAP formulations introduced previously.

\subsubsection{Admissible moving domains and geometric velocities.}\label{sec:admissible}

We adopt a shape-calculus perspective in which the evolving geometry is
initially treated as an independent geometric variable. More precisely, we
introduce a class of admissible moving domains together with associated
geometric velocity fields that parameterise the interface evolution
independently of the Darcy velocity appearing in
\eqref{eq:Darcy1}. This temporary decoupling of the geometry and PDE dynamics is introduced
purely at the variational level. The kinematic relation between the
interface velocity and the Darcy flux is subsequently reintroduced weakly
through the variational formulation developed later. The resulting
framework provides a convenient setting for shape differentiation of the
moving boundary problem and for the derivation of the corresponding
linearised and adjoint systems.

Recall that $D$ denotes the hold-all spatial domain and $\Omega(t) \subseteq D$ the evolving saturated region. Its boundary, $\Gamma(t) := \partial \Omega(t)$, consists of a fixed inlet segment $\partial D_I$, a moving boundary $\Upsilon(t) := \Gamma(t)\cap D$ driven by the kinematic condition \eqref{eq:Darcy1}, and the relevant sections of the sealed and outlet boundaries, denoted respectively by
\begin{align}
    \partial D_N^t := \Gamma(t)\cap \partial D_N, \quad \partial D_0^t := \Gamma(t)\cap \partial D_0. \label{eq: sealed_outlet_t}
\end{align}

Let \(\mathcal A\) be the family of admissible spatial domains consisting
of open subsets of \(D\) whose boundaries are piecewise \(C^{1,1}\).
In particular, the moving interface \(\Upsilon(t)\) associated with each
\(\Omega(t)\in\mathcal A\) is assumed to be a \(C^{1,1}\) curve,
possibly consisting of finitely many connected components. Next, let \(\mathcal A_T\) denote the class of admissible space--time
domains whose evolving interfaces are described through a
time-dependent family of embeddings of a fixed reference interface.
Specifically, \(\Omega_T\in\mathcal A_T\) if:
\begin{enumerate}
\item
it admits a representation
\[
\Omega_T
=
\{(x,t)\in D\times[0,T]:x\in\Omega(t)\},
\]
where \(\Omega(t)\in\mathcal A\) for each \(t>0\);
\item
there exist a fixed reference interface \(\widehat\Upsilon\) and a
family of \(C^{1,1}\)-embeddings
\[
\varphi_t:\widehat\Upsilon\longrightarrow\overline D,
\qquad t\in[0,T],
\]
which are bi-Lipschitz onto their images, with bi-Lipschitz constants
uniformly bounded for \(t\in[0,T]\), and such that
\[
t\longmapsto\varphi_t
\in
W^{1,1}\bigl(
0,T;
C^{0,1}(\widehat\Upsilon;\mathbb R^2)
\bigr).
\]
We use the continuous representative of this map on \([0,T]\), which
satisfies
\[
\varphi_0(\widehat\Upsilon)
=
\Upsilon(0)
=
\partial D_I,
\qquad
\partial D_0^0=\varnothing,
\]
and
\[
\varphi_t(\widehat\Upsilon)
=
\Upsilon(t)\cup\partial D_0^t,
\qquad t\in[0,T].
\]
\end{enumerate}

Since
\(
t\longmapsto\varphi_t
\in
W^{1,1}\bigl(
0,T;
C^{0,1}(\widehat\Upsilon;\mathbb R^2)
\bigr)
\), its time derivative exists for almost every \(t\in(0,T)\). We define
the pullback interface velocity by
\[
\widehat{\mathbf w}(X,t)
:=
\partial_t\varphi_t(X),
\qquad
X\in\widehat\Upsilon,
\]
so that
\[
\widehat{\mathbf w}
\in
L^1\bigl(
0,T;
C^{0,1}(\widehat\Upsilon;\mathbb R^2)
\bigr).
\]
The corresponding Eulerian interface velocity is defined by
\[
\mathbf w(\varphi_t(X),t)
=
\widehat{\mathbf w}(X,t)
\]
for almost every \(t\in(0,T)\). Its normal component on the moving front
is
\[
W:=\mathbf w\cdot n
\qquad\text{on }\Upsilon(t).
\]
Since tangential motion of the moving front changes only its
parameterisation, we may, when deriving the shape sensitivities, choose
the parameterisation on \(\Upsilon(t)\) so that
\[
\mathbf w=Wn.
\]
On the portion of the parameterised interface lying on the fixed outlet
boundary, the interface velocity is tangential to \(\partial D_0\), and
therefore
\[
\mathbf w\cdot n_D=0
\qquad\text{on }\partial D_0^t,
\]
where \(n_D\) denotes the outward unit normal to the fixed hold-all
domain \(D\).

When required,
the restriction of \(\mathbf w\) to \(\Upsilon(t)\) may be extended to
a tubular neighbourhood of the moving interface.

The scalar field \(W\) is introduced as the geometric normal velocity of
the moving front associated with the admissible space--time domain
\(\Omega_T\). At this stage, it is not constrained to coincide with the
Darcy velocity \(V\) defined in \eqref{eq:Darcy1}; the coupling between
the geometric evolution and the Darcy flow is instead imposed weakly in
the variational formulation below.

An immediate consequence of the present formulation is that
\(\Upsilon(t)\cup\partial D_0^t\) remains bi-Lipschitz equivalent to the
fixed reference interface \(\widehat\Upsilon\) throughout the evolution.
Since $\varphi_0(\widehat\Upsilon)=\Upsilon(0)$, it is therefore also topologically equivalent to the initial interface
\(\Upsilon(0)\). In particular, splitting, merging, or
self-intersection of the parameterised interface cannot occur. In the context of RTM, this prevents the explicit modelling of
phenomena such as dry spots \cite{Advani,MichaelMinho}. The
assumed \(C^{1,1}\) regularity of $\Upsilon(t)$ also excludes the formation of geometric
singularities such as corners or cusps and therefore limits, at least at
the theoretical level, the applicability of the framework to highly
irregular mould geometries. Nevertheless, many practically relevant defects arise from non-uniform
resin distribution or spatial variations in permeability rather than from
topological changes in the flow domain itself. As
demonstrated later in the numerical examples, the proposed inference
framework remains effective even when applied to data generated from
scenarios in which such prohibited interface behaviours may occur.

\subsubsection{Very weak formulation}

Our aim is now to introduce a mixed variational formulation of
\eqref{eq:Darcy0}--\eqref{eq:Darcy2E} in terms of the pressure field
\(p\), the moving boundary velocity \(V\), and the space--time evolution
\(\Omega_{T}\in\mathcal{A}_T\). To this end, we define the spaces
\begin{align*}
\mathcal P
&:=
L^2(0,T;H^1(D))
\cap L^\infty(D_T),\\
\mathcal P_{\mathrm{amb}}
&:=
L^2(D_T),\\
H^2_{0,\partial D}(D)
&:=
\left\{
\psi\in H^2(D)
\,:\,
\psi=0
\text{ on }
\partial D_I\cup\partial D_0,
\quad
\nabla\psi\cdot n_{D} =0
\text{ on }
\partial D_N
\right\},
\\
\Lambda
&:=L^\infty([0,T];H^2_{0,\partial D}(D)).
\end{align*}
Although the forward pressure \(p\) is sought in \(\mathcal P\), the
pressure component of the parameter-to-state map will be viewed in the
weaker ambient space \(\mathcal P_{\mathrm{amb}}\) when taking
derivatives with respect to the parameter. This distinction is required
because the Eulerian derivative of the hold-all-domain pressure extension
may possess different one-sided traces across the moving interface and
therefore need not belong to \(H^1(D)\), even though
\(p(\cdot,t)\in H^1(D)\) for almost every \(t\in(0,T)\) at each fixed
parameter field.

Furthermore, for each admissible evolution
\(\Omega_T\in\mathcal A_T\), define the moving space--time interface
\[
\Sigma_\Omega
:=
\bigcup_{t\in(0,T)}
\Upsilon(t)\times\{t\},
\qquad
\Upsilon(t):=\partial\Omega(t)\cap D.
\]

Given the interface parameterisation \(\varphi_t\) associated with
\(\Omega_T\), we define
\begin{align*}
\mathcal V(\Omega_T)
:=
\Bigg\{
q:\Sigma_\Omega\to\mathbb R
\;:\;&
\text{there exists }
Q:
\bigcup_{t\in(0,T)}
\bigl(\Upsilon(t)\cup\partial D_0^t\bigr)\times\{t\}
\longrightarrow\mathbb R
\\
&\text{such that }
Q|_{\Sigma_\Omega}=q
\text{ and }
\widehat Q(X,t):=Q(\varphi_t(X),t)\in
L^1\bigl(0,T;H^{3/2}(\widehat\Upsilon)\bigr)
\Bigg\},
\\[0.8em]
\mathcal K(\Omega_T)
:=
\Bigg\{
q:\Sigma_\Omega\to\mathbb R
\;:\;&
\text{there exists }
Q:
\bigcup_{t\in(0,T)}
\bigl(\Upsilon(t)\cup\partial D_0^t\bigr)\times\{t\}
\longrightarrow\mathbb R
\\
&\text{such that }
Q|_{\Sigma_\Omega}=q
\text{ and }
\widehat Q(X,t):=Q(\varphi_t(X),t)\in
H^1\bigl(0,T;H^{3/2}(\widehat\Upsilon)\bigr)
\Bigg\}.
\end{align*}
These spaces are equipped with the quotient norms
\[
\|q\|_{\mathcal V(\Omega_T)}
:=
\inf_{\substack{Q|_{\Sigma_\Omega}=q}}
\|\widehat Q\|_
{L^1(0,T;H^{3/2}(\widehat\Upsilon))}
\]
and
\[
\|q\|_{\mathcal K(\Omega_T)}
:=
\inf_{\substack{Q|_{\Sigma_\Omega}=q}}
\|\widehat Q\|_
{H^1(0,T;H^{3/2}(\widehat\Upsilon))},
\]
where the infima are taken over all admissible extensions \(Q\)
appearing in the definitions above.

The extension \(Q\) is used only to express time-dependent regularity
through pullback to the fixed reference interface
\(\widehat\Upsilon\). The physical element
\(q\in\mathcal V(\Omega_T)\) or
\(q\in\mathcal K(\Omega_T)\) is defined only on the moving front
\(\Sigma_\Omega\); the values of \(Q\) on \(\partial D_0^t\) are
auxiliary.

To formulate the moving boundary problem in a form suitable for shape
differentiation and adjoint analysis, we employ a very weak formulation
inspired by the work of \cite{van2010goal} on free-boundary optimisation
problems. Following this approach, the Darcy equation
\eqref{eq:Darcy0} is tested against \(\lambda\in\Lambda\) and integrated
twice by parts over the evolving space--time domain
\(\Omega_T\). The no-flow condition on \(\partial D_N\) and the
kinematic relation \eqref{eq:Darcy1} are incorporated weakly through the
resulting boundary terms, introducing the interface velocity \(V\)
explicitly into the formulation. The Dirichlet conditions on
\(\partial D_I\), \(\partial D_0\), and \(\Upsilon(t)\) are imposed
weakly through an extension \(p_B\) of the prescribed boundary data. This
leads to the following mixed very weak formulation of
\eqref{eq:Darcy0}--\eqref{eq:Darcy2E}.

\begin{definition}[Very weak formulation]
\label{def:weak}
Given \(u\in E\), find \((p,V,\Omega_{T})\) such that
\[
p\in\mathcal P,
\qquad
V\in\mathcal V(\Omega_{T}),
\qquad
\Omega_{T}\in\mathcal{A}_T,
\]
and
\begin{align}
&
-\int_0^T\int_{\Omega(t)}
p\,\nabla\cdot(e^u\nabla\lambda)\,dx\,dt
+
\int_0^T\int_{\partial\Omega(t)}
p_B e^u\nabla\lambda\cdot n\,ds\,dt
\nonumber\\
&\quad
+
\mu_f\phi
\int_0^T\int_{\Upsilon(t)}
\lambda V\,ds\,dt
+
\int_0^T\int_{D\setminus\overline{\Omega}(t)}
(p-p_0)\lambda\,dx\,dt
\nonumber\\
&\quad
+
\int_0^T\int_{\Upsilon(t)}
(W-V)\kappa\,ds\,dt
=
0
\label{eq:weak1}
\end{align}
for all
\(
(\lambda,\kappa)\in\Lambda\times\mathcal K(\Omega_{T}).
\)

Here, for each admissible evolution
\(\Omega_T\in\mathcal A_T\), we assume that, for almost every
\(t\in(0,T)\), there exists a sufficiently regular lifting
\(p_B(\cdot,t)\in H^2(D)\) of the prescribed Dirichlet boundary data,
satisfying, in the sense of traces,
\[
p_B=p_I
\quad\text{on }\partial D_I,
\qquad
p_B=p_0
\quad\text{on }\partial D_0\cup\Upsilon(t).
\]
We further assume that
\[
p_B\in L^1(0,T;H^2(D))
\]
and that, for every \(\delta>0\),
\[
p_B\in L^\infty(\delta,T;H^2(D)).
\]
\end{definition}

The term integrated over \(D\setminus\overline{\Omega}(t)\) in \eqref{eq:weak1} weakly enforces the
condition \(p=p_0\) in the unsaturated region, while the final term
imposes the kinematic compatibility between the geometric normal velocity
\(W\) and the Darcy velocity \(V\) on the moving interface. For fixed \(u\), \(V\), and \(\Omega_T\), the residual in
\eqref{eq:weak1} also admits a continuous extension with respect to the
pressure variable from \(\mathcal P\) to
\(\mathcal P_{\mathrm{amb}}\). Indeed, the pressure appears only in
volume integrals against functions in \(L^2(D)\). This extension is used
below when defining the Eulerian pressure derivative.

All time integrals in the very weak formulation are defined over the
fixed interval \((0,T)\). Interface integrals are understood to vanish
after complete filling, once $\overline{\Omega}(t)=\overline D$. For a fixed reference solution \(\Omega_T\), these terms may therefore
be written equivalently over the active-interface interval
\((0,T^*)\), where
\begin{equation}
T^*:=\min\{T,\tau^*\},
\label{eq:tstar}
\end{equation}
and \(\tau^*\) is the filling time associated with that reference
solution. Throughout the linearisation and adjoint derivations below,
the residual is differentiated in its original fixed-horizon form over
\((0,T)\). The notation \(T^*\) is used only to indicate that the
reference-state interface contributions vanish for \(t\geq T^*\);
\(T^*\) is not treated as a variable integration limit in these
derivations.

\begin{remark}
For each admissible space--time domain
\(
\Omega_{T}\in\mathcal A_T,
\)
the associated geometric normal velocity \(W\) is determined by the
interface parameterisation \(\varphi_t\) associated with \(\Omega_T\). Consequently, \(W\) is not treated as
an independent state variable of the very weak formulation.
\end{remark}
\begin{remark}\label{rem:smoothness}
The very weak formulation requires sufficient regularity of the coefficient
\(e^u\) for the term
\(\nabla\cdot(e^u\nabla\lambda)\) to be well defined. It is sufficient to
assume \(u\in W^{1,\infty}(D)\). For the Matérn prior considered here,
the CM space satisfies
\(E\simeq H^{\nu+d/2}(D)\), and Sobolev embedding yields
\(E\hookrightarrow W^{1,\infty}(D)\) whenever \(\nu>1\). 
If stronger geometric regularity of the evolving interfaces is required,
one may impose higher smoothness by choosing larger values of \(\nu\).
Establishing regularity properties of the coupled moving boundary
evolution itself, however, lies beyond the scope of the present work.
\end{remark}

The very weak formulation introduced above is primarily motivated by the
derivation of linearisations, adjoint systems, and representer formulae
required for variational inverse methods. While the formulation provides a
natural setting for shape differentiation and sensitivity analysis, a
complete well-posedness theory for the associated coupled moving boundary
problem remains unavailable in the present setting. Accordingly, throughout
the remainder of this work we make the following assumption.

\begin{assumption}
\label{asu:1}
For every \(u\in E\), the very weak formulation from
Definition~\ref{def:weak} admits a unique solution.
\end{assumption}

\subsubsection{Geometric perturbations of the identity.}\label{sec:geometric_perturbations}

To conduct shape sensitivity analysis of the very weak formulation, we introduce admissible one-sided geometric perturbations through
time-dependent vector fields \cite{delfour2011shapes}. We consider
perturbation fields
\(
\deltao \in \Theta_T := C^{0,1}([0,T];\Theta),
\)
where $\Theta$ is the admissible cone
\begin{align}
\Theta
:=
\left\{
\deltao\in C^{0,1}(\overline D;\mathbb R^2)
\,\middle|\,
\deltao\cdot n_D=0
\text{ on }\partial D_I\cup\partial D_N,
\quad
\deltao\cdot n_D\leq0
\text{ on }\partial D_0
\right\}.
\label{eq:admissible_vec}
\end{align}
The boundary conditions imposed on \(\Theta\) ensure that the associated
domain perturbations preserve the fixed inlet and no-flow boundaries,
while remaining contained within the hold-all domain \(D\).

The equality constraints in the definition of \(\Theta\) ensure that
the generated perturbation flow preserves the fixed inlet and no-flow
boundaries. The one-sided condition on \(\partial D_0\) prevents the
perturbed evolution from leaving the hold-all domain \(D\), while
allowing both tangential motion along the outlet and inward motion into
the interior of \(D\). Consequently, \(\Theta\) is an admissible cone
rather than a vector space.

For each fixed \(t\in[0,T]\), a perturbation field
\(\deltao(\cdot,t)\in\Theta\) generates a one-sided flow of
transformations \(T_{s\deltao}(\cdot,t)\) defined by
\[
\frac{\partial}{\partial s}
T_{s\deltao}(x,t)
=
\deltao\bigl(T_{s\deltao}(x,t),t\bigr),
\qquad
T_{0\deltao}(x,t)=x,
\qquad
s\in[0,\bar s].
\]
In particular,
\[
\left.
\frac{\partial}{\partial s}
T_{s\deltao}(x,t)
\right|_{s=0^+}
=
\deltao(x,t).
\]
The corresponding family of perturbed domains is defined by
\begin{equation}
\Omega_s(t)
:=
T_{s\deltao}(\Omega(t)),
\qquad
s\in[0,\bar s].
\label{eq:deltao_family}
\end{equation}
The vector field \(\deltao\) therefore represents the infinitesimal
geometric variation of the moving domain \(\Omega(t)\). Since the
outlet constraint is one-sided, all geometric derivatives below are
understood as directional derivatives at \(s=0^+\).

For any \(\Omega_T\in\mathcal A_T\), the parameterised interface
satisfies
\[
\Upsilon(t)\cup\partial D_0^t
=
\varphi_t(\widehat\Upsilon).
\]
Under the perturbation \eqref{eq:deltao_family}, the parameterised
interface is given by
\[
\bigl(\Upsilon(t)\cup\partial D_0^t\bigr)_s
=
T_{s\deltao}
\bigl(\Upsilon(t)\cup\partial D_0^t\bigr)
=
T_{s\deltao}\circ\varphi_t(\widehat\Upsilon).
\]

For sufficiently small \(\bar s>0\), the flow maps
\(T_{s\deltao}(\cdot,t)\) are bi-Lipschitz for every
\(s\in[0,\bar s]\) and \(t\in[0,T]\)
\cite{delfour2011shapes,van2010goal}. The boundary conditions defining
\(\Theta\) ensure that these maps preserve the fixed inlet and no-flow
boundaries and do not move points outside \(D\).

Define the perturbed interface parameterisation by
\[
\varphi_{t,s}(X)
:=
T_{s\deltao}\bigl(\varphi_t(X),t\bigr),
\qquad
X\in\widehat\Upsilon.
\]
Since
\(
t\longmapsto\varphi_t
\in
W^{1,1}\bigl(
0,T;
C^{0,1}(\widehat\Upsilon;\mathbb R^2)
\bigr)
\)
and
\(
\deltao\in C^{0,1}([0,T];\Theta),
\)
the perturbed parameterisation satisfies
\[
t\longmapsto\varphi_{t,s}
\in
W^{1,1}\bigl(
0,T;
C^{0,1}(\widehat\Upsilon;\mathbb R^2)
\bigr)
\]
for every sufficiently small \(s\in[0,\bar s]\). Hence the perturbed
evolutions retain both the required spatial bi-Lipschitz regularity and
the required absolute continuity in time.

\begin{remark}
The extension $p_{B}$ used in Definition~\ref{def:weak} may be chosen locally using tubular neighbourhoods of the
corresponding boundary components, since the inlet and the moving front
remain disjoint for \(t>0\). We additionally choose \(p_B\) constant along
normal directions in a tubular neighbourhood of
\(\Upsilon(t)\cup\partial D_0\). Consequently, under
admissible interface perturbations, the material derivative of
\(p_B\) vanishes in this neighbourhood, and no additional shape
derivative terms arise from the lifting in the linearisation of
\eqref{eq:weak1}.
\end{remark}

\subsubsection{Formal linearisation.}

Having introduced the admissible geometric perturbations, we now define
the space of Eulerian pressure perturbations. Although the forward
pressure belongs to \(\mathcal P\), its Eulerian derivative is understood
in the weaker ambient space \(\mathcal P_{\mathrm{amb}}\). In particular,
the derivative is zero in the bulk unsaturated region, where the pressure
extension is the fixed constant \(p_0\), but its wet-side trace on the
moving interface need not vanish. We therefore define
\begin{align*}
\mathcal P_0(\Omega_T)
:=
\Bigg\{
q\in L^2(D_T)
\;:\;&
q|_{\Omega(t)}\in H^1(\Omega(t))
\quad\text{for a.e. }t\in(0,T),
\\
&
q=0
\quad\text{a.e. in }
D\setminus\overline{\Omega(t)},
\\
&
q=0
\quad\text{on }
\partial\Omega(t)\cap
\bigl(\partial D_I\cup\partial D_0\bigr),
\\
&
\int_0^T
\|q(\cdot,t)\|_{H^1(\Omega(t))}^2\,dt
<\infty
\Bigg\}.
\end{align*}
Here, the zero value in the unsaturated region is understood as an
\(L^2(D)\)-extension. No matching condition is imposed between this
extension and the trace of \(q|_{\Omega(t)}\) on the moving interface
\(\Upsilon(t)\).

We now derive the formal linearisation of the very weak formulation with
respect to perturbations of the permeability, pressure, interface
velocity, and geometry.
\begin{proposition}[Formal one-sided directional linearisation of the very weak formulation]
\label{prop:LinearisedPhysics}
Let \(u\in E\), and let $z=(p,V,\Omega_{T})$ be the corresponding solution to the very weak form of Definition~\ref{def:weak}. Let $\tau^{*}$ be the corresponding filling time and $T^*$ be defined via \eqref{eq:tstar}. Then the formal one-sided directional linearisation of
\eqref{eq:weak1} at \((u,z)\), in the admissible direction
\[
(h,\deltap,\delta V,\deltao)
\in
E\times\mathcal P_0(\Omega_T)
\times\mathcal V(\Omega_T)\times\Theta_T,
\]
is given by
\begin{align}
0
&=
-\int_0^T\int_{\Omega(t)}
p\,\nabla\cdot(e^u h\nabla\lambda)
\,dx\,dt
+
\int_0^T\int_{\partial\Omega(t)}
p_Be^uh\nabla\lambda\cdot n
\,ds\,dt
\nonumber\\
&\quad
-
\int_0^T\int_{\Omega(t)}
\deltap\,\nabla\cdot(e^u\nabla\lambda)
\,dx\,dt +
\int_0^{T^*}\int_{\Upsilon(t)}
\bigl(\mu_f\phi\lambda-\kappa\bigr)
\delta V
\,ds\,dt
\nonumber\\
&\quad
+
\mu_f\phi
\int_0^{T^*}\int_{\Upsilon(t)}
V(\nabla\lambda\cdot n+c\lambda)\deltao_n
\,ds\,dt
\nonumber\\
&\quad
+
\int_0^{T^*}\int_{\Upsilon(t)}
\kappa
\frac{D}{Dt}(\deltao_n)
\,ds\,dt .
\label{eq:linearised_weak}
\end{align}
for all
\(
(\lambda,\kappa)\in\Lambda\times\mathcal K(\Omega_{T})
\)
where
\(
\deltao_n:=\deltao\cdot n,
\)
\(D/Dt\) denotes the material derivative associated with
the interface velocity field \(\mathbf w\), and
\(
c=\operatorname{div}_{\Upsilon}n
\)
is the curvature of \(\Upsilon(t)\).

\begin{proof}
See Appendix~\ref{app:LinearisedPhysics}.
\end{proof}
\end{proposition}

Given a perturbation direction \(h\in E\), the corresponding formal
first-order state variations are characterised by the variational
system \eqref{eq:linearised_weak}. Their relation to derivatives of the
actual parameter-to-observable map, $\GG$, is specified in
Assumption~\ref{ass:state_diff}.

The principal advantage of the very weak formulation becomes apparent at
the linearisation level. Shape differentiation of moving-domain integrals
naturally generates boundary contributions involving both the traces of
the test functions and their normal derivatives. In standard weak
formulations, such terms do not combine directly into the classical
linearised interface conditions. By contrast, after integrating the Darcy
equation twice by parts, the present formulation reorganises the boundary
terms so that the coefficient of
\(
\nabla\lambda\cdot n
\)
on the moving interface contains the linearised Dirichlet condition
\[
\deltap+(\nabla p\cdot n)\deltao_n.
\]
Consequently, after integrating the variational system by parts, this
coefficient can be identified with the strong moving boundary condition
\[
\deltap+(\nabla p\cdot n)\deltao_n=0
\qquad\text{on }\Upsilon(t).
\]
Similarly, the remaining interface contributions yield the corresponding
linearised flux and kinematic conditions. Thus, the very weak formulation
does not eliminate the geometric boundary terms; rather, it reorganises
them into a structure from which the strong linearised moving boundary
system can be recovered consistently. Indeed, formally integrating the variational system
\eqref{eq:linearised_weak} by parts and eliminating the auxiliary
interface variable \(\delta V\) yield the following strong form of the
linearised moving boundary problem:
\begin{align}
-\nabla \cdot \big[e^{u}h\nabla p + e^{u}\nabla (\delta p)\big]
&= 0,
&&x \in \Omega(t),\quad t>0,
\label{eq: LinearisedStateStart}
\\
\delta p
&= 0,
&&x \in D\backslash\overline{\Omega } (t),\quad t\geq 0,
\label{eq: LinearisedState5}
\\
\delta p + (\nabla p\cdot n)\deltao_n
&= 0,
&&x \in \Upsilon(t),\quad t < T^*,
\label{eq: LinearisedState2}
\\
\delta p
&= 0,
&&x \in \partial D_I\cup\partial D_0,\quad t \geq 0,
\label{eq: LinearisedState3}
\\
\nabla(\delta p)\cdot n_{D}
&= 0,
&&x \in \partial D_N,\quad t \geq 0,
\label{eq: LinearisedState4}
\\
\frac{D(\deltao_n)}{Dt}
+
\frac{1}{\mu_f\phi}
\big[e^{u}h\nabla p + e^{u}\nabla (\delta p)\big]\cdot n
+
cV\deltao_n
&= 0,
&&x \in \Upsilon(t),\quad t < T^*,
\label{eq: LinearisedStateEnd_minus1}
\\
\deltao_n(x,0)
&= 0,
&&x \in \Upsilon(0)=\partial D_I.
\label{eq: LinearisedStateEnd}
\end{align}
The condition \eqref{eq: LinearisedState5} records the zero extension
of the linearised pressure to the unsaturated region encoded in
\(\mathcal P_0(\Omega_T)\); it is not an additional equation for the
wet-region pressure perturbation. In particular, no matching condition
is imposed between the zero exterior extension and the wet-side trace of
\(\delta p\) on \(\Upsilon(t)\).

\subsubsection{Adjoint-based representers.} 
We now derive a reduced representation of the Fréchet derivative of the
observation map \(\mathcal G\) using adjoint techniques. As discussed in
Section~\ref{sec:reduced}, such representations are central to the
efficient implementation of the LM scheme and the LMAP
posterior approximation, as they allow the gradient and Hessian
information to be computed without explicitly solving the linearised
state equations for each perturbation direction.

To this end, we first require differentiability of the forward map with
respect to the log-permeability field. This is formalised in the
following assumption.

\begin{assumption}[Differentiability and sensitivity characterisation]
\label{ass:state_diff}
For each \(u\in E\), let
\(
z[u]=(p[u],V[u],\Omega_T[u])
\)
denote the weak solution of Definition~\ref{def:weak}. Although
\(p[u]\in\mathcal P\) for every \(u\), we regard the hold-all-domain
pressure map as taking values in the weaker ambient space
\(\mathcal P_{\mathrm{amb}}\).

We assume that the map
\[
u\longmapsto p[u],
\qquad
E\longrightarrow\mathcal P_{\mathrm{amb}},
\]
is Fr\'echet differentiable. For every \(h\in E\), we denote its
derivative by
\(
\deltap[h]:=Dp[u]h
\in\mathcal P_0(\Omega_T)\).
We further assume that there exist associated directional interface variations and admissible
one-sided geometric variations
\[
\delta V[h]\in\mathcal V(\Omega_T),
\qquad
\deltao[h]\in\Theta_T,
\]
such that
\(
\bigl(\deltap[h],\delta V[h],\deltao[h]\bigr)
\)
is characterised by the formal directional linearisation
\eqref{eq:linearised_weak}. The geometric variation
\(\deltao[h]\) is understood as an admissible one-sided Eulerian shape
variation at \(s=0^+\).
\end{assumption}

Under Assumption~\ref{ass:state_diff}, and since
\(\mathcal H_i\in L^\infty(D_T)\), the map
\(\mathcal G:E\to\mathbb R^M\) is Fr\'echet differentiable, with
\begin{equation}\label{eq:forconsistency}
    D\mathcal G_i(u)h
=
\int_0^T\int_D
\mathcal H_i(x,t)\,\deltap[h](x,t)\,dx\,dt
=
\int_0^T\int_{\Omega(t)}
\mathcal H_i(x,t)\,\deltap[h](x,t)\,dx\,dt.
\end{equation}
Since \(\deltap[h]\in\mathcal P_0(\Omega_T)\),
its hold-all-domain extension vanishes almost everywhere in
\(D\setminus\Omega(t)\). This yields the second equality in
\eqref{eq:forconsistency}.

Assumption~\ref{ass:state_diff} reflects a common approach in computational
PDE-constrained inversion, where the sensitivities required for
adjoint-based optimisation are obtained by formally linearising the state
equations and deriving the corresponding adjoint equations
\cite{WORTHEN201423,Petra_Zhu_Stadler_Hughes_Ghattas_2012,Petra2,Puel}.
In the present moving-boundary setting, a rigorous differentiability
theory for the parameter-to-state map associated with the coupled
two-dimensional Darcy--interface evolution problem is not currently
available and lies beyond the scope of this work. The formal sensitivity
framework developed here is nevertheless consistent with the
one-dimensional setting studied in Section~\ref{sec:1D}, where explicit
solutions are available and the Fréchet differentiability of the
parameter-to-observable map is established rigorously.

\begin{remark}[Consistency with the 1D formulation]
\label{rem: consistency1D}
Upon reduction of
\eqref{eq: LinearisedStateStart}--\eqref{eq: LinearisedStateEnd}
to 1D, where
\(D=(0,x^*)\) and geometric quantities such as curvature and the unit
normal simplify, the linearised state admits the explicit solution
\begin{align}
\delta p[h](x,t)
&=
\mathbb{I}_{\{x \le \Upsilon (t)\}}
\frac{p_I-p_0}{F[u](\Upsilon (t))^2}
\Bigg[
F[u](\Upsilon (t))
\int_0^x e^{-u}h d \xi
-
F[u](x)
\int_0^{\Upsilon (t)}e^{-u}h d \xi
\nonumber\\
&\quad
+
\mathbb{I}_{\{t \le \tau^*\}}
F[u](x)e^{-u (\Upsilon (t))}
\deltao[h](t)
\Bigg],
\nonumber\\
\deltao[h](t)
&=
\frac{1}{F[u](\Upsilon (t))}
\int_0^{\Upsilon (t)}
\int_0^\xi e^{-u}h d z d \xi.
\end{align}
Substituting these expressions into the right-hand side of
\eqref{eq:forconsistency} yields exactly the Fréchet derivative obtained
rigorously in Theorem~\ref{the:G_differentiable}. This provides a
consistency check for the proposed 2D variational and
shape-differentiation framework.
\end{remark}

Under Assumption~\ref{ass:state_diff}, the linearised system
\eqref{eq:linearised_weak} provides a state-sensitivity
characterisation of the Fr\'echet derivative of the observation map
through the perturbation variables
\((\deltap,\delta V,\deltao)\). However, directly solving the linearised system for every perturbation
direction \(h\in E\) would be computationally prohibitive within the
iterative optimisation framework introduced earlier. The purpose of the
adjoint formulation is therefore to eliminate the linearised state
variables and obtain a reduced representation of the Fréchet derivative
through a collection of observation-dependent representers.

\begin{theorem}[Adjoint characterisation of representers]
\label{thm:adjoint_rep_2D}
Let \(u\in E\) and let
\(
z=(p,V,\Omega_{T})
\)
be the corresponding solution to the very weak formulation of
Definition~\ref{def:weak} and $T^*$ defined by \eqref{eq:tstar}. Suppose that Assumption~\ref{ass:state_diff} holds. For each \(i=1,\ldots,M\), let
\[
(\lambda_i,\kappa_i)
\in
\left\{
(\lambda,\kappa)\in\Lambda\times\mathcal K(\Omega_T)
:
\kappa(\cdot,T^*)=0
\right\}
\]
denote an adjoint pair satisfying
\begin{align}
&-
\int_0^T\int_{\Omega(t)}
\deltap\,
\nabla\cdot\bigl(e^u\nabla\lambda_i\bigr)
\,dx\,dt
+
\int_0^{T^*}\int_{\Upsilon(t)}
(\mu_f\phi\lambda_i-\kappa_i)\delta V
\,ds\,dt
\nonumber\\
&\quad
-
\int_0^{T^*}\int_{\Upsilon(t)}
\left[
\frac{D\kappa_i}{Dt}
+
cV(\kappa_i-\mu_f\phi\lambda_i)
-
\mu_f\phi V\nabla\lambda_i\cdot n
\right]
\deltao_n
\,ds\,dt
\nonumber\\
&=
\int_0^T\int_{\Omega(t)}
\mathcal H_i(x,t)\,\deltap(x,t)
\,dx\,dt
\label{eq:adjoint}
\end{align}
for all
\(
(\deltap,\delta V,\deltao)
\in
\mathcal P_0(\Omega_T)
\times\mathcal V(\Omega_T)\times\Theta_T\). Assume additionally that the adjoint solution has sufficient regularity
such that
\(
\nabla\lambda_i
\in
L^\infty(D_T)
\). Then the Fréchet derivative of the observation functional
\(\mathcal G_i\) admits the representation
\[
D\mathcal G_i(u)h
=
\langle r_i[u],h\rangle_{L^2(D)},
\qquad h\in E,
\]
where
\begin{align}
\label{eq: adj_Representer_i}
r_i[u](x)
:=
-
e^{u(x)}
\int_0^T
\mathbb I_{\Omega(t)}(x)\,
\nabla p[u](x,t)\cdot
\nabla\lambda_i(x,t)\,dt .
\end{align}
Under the stated regularity assumption,
\begin{align*}
\|r_i[u]\|_{L^2(D)}
&\leq
\|e^u\|_{L^\infty(D)}
\|\nabla\lambda_i\|_{L^\infty(D_T)}
T^{1/2}
\|\nabla p[u]\|_{L^2(D_T)}.
\end{align*}
Consequently, \(r_i[u]\in L^2(D)\).
\begin{proof}
See Appendix~\ref{app: Representers2D}.
\end{proof}
\end{theorem}
The representer \(r_i[u]\) encodes the sensitivity of the \(i\)-th
observation functional with respect to perturbations of the
log-permeability field. Importantly, once the adjoint variables
\((\lambda_i,\kappa_i)\) have been computed, the Fréchet derivative is
evaluated through a single \(L^2(D)\) inner product, thereby avoiding the
explicit solution of the linearised moving boundary problem for each
parameter perturbation direction \(h\).

\begin{remark}
The adjoint system \eqref{eq:adjoint} can alternatively be derived by
expressing the LM functional
\eqref{eq:LM_cost} in terms of the linearised state variable
\(\delta p\) (i.e.\ via \eqref{eq:forconsistency}) and constructing a
Lagrangian based on the very weak formulation of the moving boundary
problem. The corresponding first-order optimality conditions then yield
an equivalent adjoint system upon elimination of the state perturbations.
While this approach is conceptually natural from an optimisation
perspective, it leads to a more involved derivation due to the weak
coupling of bulk and interface variables and the presence of geometric
terms. For completeness, this alternative route is detailed in
\cite{MCthesis}.
\end{remark}

\subsubsection{Numerical implementation and computational costs.}\label{sec:2Dimple}
While the forward model is introduced through the very weak variational
formulation in Section~\ref{sec:2D}, the numerical solution of the moving
boundary problem is carried out using the RTM-Flow MATLAB tool
\cite{Zenodo}, which implements a control volume finite element method
(CVFEM) following \cite{MichaelMinho,Advani,discontinuity_CVFEM}. At each iteration of Algorithm~\ref{alg:LM}, the current estimate of the
log-permeability field \(u\) is used to solve the forward moving boundary
problem. The CVFEM solver provides numerical approximations of the
pressure field \(p[u]\), the Darcy velocity, the moving interface
\(\Upsilon[u]\), the saturated domain \(\Omega[u]\), and the associated
filling time \(\tau^*[u]\) (and hence \(T^*[u]\)). These quantities are
then used to evaluate the observation operator and the
Onsager--Machlup functional, construct the adjoint representers, and
assemble the reduced LM update.

For the numerical experiments, observations are evaluated at the
prescribed sensor locations and observation times. In the adjoint
discretisation, the spatial Dirac mass at a sensor is represented by the
finite-element nodal load vector obtained from the barycentric
coordinates of the sensor within its containing element, while temporal
localisation is implemented using a Gaussian mollifier on the time
grid. This provides a discrete counterpart of the regularised
observation functionals introduced in Section~2.1. Further
implementation details are provided in the accompanying code.

The numerical approximations of
\(
p[u], \Upsilon[u], \Omega[u], \tau^*[u]
\)
provided by RTM-Flow constitute all quantities required for the solution
of the adjoint equations and the computation of the representers from
\eqref{eq: adj_Representer_i}. Rather than discretising the very weak
formulation directly, however, we employ the adjoint equations in
their corresponding strong form. Indeed, from the weak adjoint
formulation \eqref{eq:adjoint}, one formally obtains the system
\begin{align}
-\nabla \cdot \big(e^{u}\nabla\lambda_i\big)
&= \mathcal{H}_i,
&&x \in \Omega(t),\ t > 0,
\label{eq: Adjoint_i_1}
\\
\lambda_i &= \kappa_i/(\mu_f\phi),
&&x \in \Upsilon(t),\ t < T^*,
\label{eq: Adjoint_i_3}
\\
\lambda_i &= 0,
&&x \in \partial D_I\cup \partial D_0,\ t \geq 0,
\label{eq: Adjoint_i_4}
\\
\nabla\lambda_i\cdot n_{D} &= 0,
&&x \in \partial D_N,\ t \geq 0,
\label{eq: Adjoint_i_5}
\\
\frac{D\kappa_i}{Dt}
+ e^{u}(\nabla p \cdot n)(\nabla\lambda_i\cdot n) &= 0,
&&x \in \Upsilon(t),\ t < T^*,
\label{eq: Adjoint_i_6}
\\
\kappa_i(x,T^*) &= 0,
&&x \in \Upsilon(T^*).
\label{eq: Adjoint_i_7}
\end{align}
These equations are solved numerically by reusing and adapting components
of the CVFEM discretisation employed for the forward problem. In
particular, the elliptic problem for \(\lambda_i\) is discretised on the
same mesh and control volumes as the forward pressure equation, while the
interface evolution equation for \(\kappa_i\) is solved along the moving
front using the same geometric representation of \(\Upsilon(t)\). Further
implementation details, including the treatment of the moving interface,
time discretisation, and coupling between bulk and interface variables,
are provided in Appendix~\ref{app:adjoint_eqns}.

The computational cost of the adjoint computations remains relatively
modest compared to the forward moving boundary solve. The forward problem
\eqref{eq:Darcy0}--\eqref{eq:Darcy2E} is nonlinear due to the coupling
between the Darcy equation and the evolving interface, requiring repeated
updates of the saturated region and moving front. In contrast, once the
forward quantities
\(
p[u], \Upsilon[u], \Omega[u]
\)
have been computed, the adjoint equations
\eqref{eq: Adjoint_i_1}--\eqref{eq: Adjoint_i_7} are linear and may be
solved independently for each observation index \(i\).

Several aspects of the implementation further reduce the
computational burden. First, stiffness matrices assembled during the
forward solve may be reused, since the interface evolution is already
known. Second, the adjoint systems differ only through the source term
\(\mathcal H_i\) in \eqref{eq: Adjoint_i_1}, making the computations
naturally parallelisable across observations. Third, many adjoint solves
may be discarded entirely: if the observation location lies outside the
saturated region, i.e.
\(
x_i\notin\Omega(t_i),
\)
then the mollifier \(\mathcal H_i\) is negligible over the solution
domain and the corresponding adjoint solution remains effectively zero. Similarly, when solving the adjoint equations backward in time, the source
term remains negligible until the temporal mollifier centred at
\(t=t_i\) becomes active. Consequently, the numerical integration need
only commence a small number of time steps before \(t_i\), further
reducing the computational cost.

An additional computational saving arises during Step~8 of
Algorithm~\ref{alg:LM}. If a proposed update \(\widehat u\) fails to
decrease the Onsager--Machlup functional, the adjoint variables and
representers need not be recomputed, since these depend only on the
current iterate \(u\). It therefore suffices to increase the damping
parameter \(\alpha\) and recompute the reduced
Levenberg--Marquardt update
\eqref{eq:MAP_update}--\eqref{eq:MAP_update2}. In the numerical implementation, this backtracking procedure is
parallelised. At each iteration, several candidate updates corresponding
to the damping parameters
\[
\alpha,\quad
\alpha/\gamma,\quad
\alpha/\gamma^2,\quad
\alpha/\gamma^3,\quad
\alpha/\gamma^4
\]
are computed simultaneously, and the associated forward moving boundary
problems are solved in parallel. The accepted update is then chosen as
the candidate yielding the smallest value of
\(J_{\mathrm{OM}}\). Although this increases the number of forward solves
per iteration, it substantially reduces idle time associated with serial
rejection of unsuccessful updates.

Finally, solving the adjoint equations and constructing the representers
require storing approximations of the forward quantities
\(
p[u],\Upsilon[u],\Omega[u]
\)
throughout the forward simulation. While this memory requirement was not
restrictive for the examples considered here, it may become significant
for large-scale three-dimensional simulations.

\section{Numerical results}\label{sec:numerics}

The numerical experiments with synthetic data are designed to assess both
the reconstruction accuracy of the MAP estimator and the quality of the
LMAP covariance approximation as a local characterisation of posterior
uncertainty. In particular, we compare the proposed reduced adjoint-based
framework with derivative-free approaches such as ensemble Kalman
inversion (EKI), while also investigating the effect of increasing
observational information over time.
The code is available at  \url{https://github.com/ChaffyHurdle/RTM-flow/tree/lmap}.

\subsection{Experimental setup}

For the principal 2D benchmark, the physical domain \(D\)
is taken to be the unit square. The annulus and fork configurations are
introduced separately in Section~\ref{sec:complex_geometries}. Fluid enters through the inlet boundary
\(\partial D_I\) located on the left edge and exits through the outlet
boundary \(\partial D_0\) on the right edge, while the top and bottom
edges form the no-flow boundary \(\partial D_N\) (see
Figure~\ref{fig: setup_eg}). In the 1D setting,
\(D=(0,1)\), with the left endpoint corresponding to the inlet and the
right endpoint to the outlet.

For all experiments, the prior measure is chosen as a Gaussian Matérn
field with hyperparameters
\[
(\sigma,l,\nu)=(0.5,0.1,1.5),
\]
and prior mean \(\bar u=0\). This choice is consistent with the
regularity requirements discussed in
Remark~\ref{rem:smoothness}, since \(\nu>1\) ensures sufficient
regularity for the very weak formulation. The prior covariance operator
is constructed through discretisation of the Matérn correlation kernel.

\subsection{Measurement locations and observation times}

The observational data are determined by a collection of
\(N_s\) spatial sensor locations
\(
\{x_i\}_{i=1}^{N_s}\subset D
\)
and \(N\) observation times
\(
\{t_n\}_{n=1}^{N}
\).
For each observation time \(t_n\), we denote by
\[
y^{(n)}\in\mathbb R^{N_s}
\]
the vector of measurements collected from all sensors at time \(t_n\).

For each \(n\), we define
\[
\mu_n
=
\mu\big(u\,|\,y^{(1)},\dots,y^{(n)}\big)
\]
as the posterior distribution conditioned on all observations available
up to time \(t_n\). The corresponding Gaussian approximation is computed
using Algorithm~\ref{alg:LM} with
\(M=nN_s\) observations:
\[
\mu_n
\approx
\mathcal N\big(
u_{\mathrm{MAP}}^{(n)},
\mathcal C_{\mathrm{MAP}}^{(n)}
\big).
\]

We emphasise that this does not correspond to a sequential Bayesian
filtering problem. Rather, for each \(n\) we solve an independent
Bayesian inverse problem with the same prior measure but with an
increasing amount of observational information as additional measurement
times are included.

In practice, measurements should be distributed throughout the moving
boundary regime \(t<\tau^*\) in order to enable online permeability
estimation while avoiding redundant observations after complete filling
of the domain. Since the filling time \(\tau^*\) depends on the unknown
permeability field, however, it is not known \emph{a priori}. We
therefore select observation times using the prior mean field
\(\bar u\) as a reference configuration.

Observation times are selected to correspond to prescribed fractions of
the nominal filling progression. Specifically, for
\(\zeta_n\in(0,1]\), we interpret \(\zeta_n\) as the fraction of the
unit filling length reached at the \(n\)th observation time. Using the
1D filling-time relation \eqref{eq:fill1D}, evaluated at the
constant prior mean \(\bar u=0\), gives
\[
t_n
=
\frac12
\frac{\mu_f\phi\,\zeta_n^2}{p_I-p_0}.
\]
The fractions \(\zeta_n\) used in the numerical experiments are specified
in Table~\ref{tab:parameters}. This choice distributes the observations
according to the progression of the filling process rather than uniformly
in time. Although the expression above is derived from the 1D
solution, it also applies to the 2D reference configuration:
for constant permeability and the uniform inlet and outlet conditions
considered here, the solution is independent of the transverse coordinate
and the filling dynamics reduce to the corresponding 1D problem.

\subsection{Generation of synthetic data}

For all experiments, the true log-permeability field is sampled from the
prior distribution,
\(
u^\dagger\sim\mu_0,
\)
and represented as piecewise constant over a fine computational mesh. The
forward model is then solved using the physical parameters
\(
(\mu_f,\phi,p_I,p_0)
\)
given in Table~\ref{tab:parameters}, generating synthetic pressure
observations through the kernels \(\mathcal H_i\) introduced in
Sections~\ref{sec:1Dimple} and \ref{sec:2Dimple}.

Observational noise is added according to
\(
\eta\sim\mathcal N(0,\Sigma),
\)
where \(\Sigma\) is chosen diagonal with entries
\[
\Sigma_{ii}
=
\big(
\chi\max_{j=1,\ldots,M} \mathcal G_j(u^\dagger)
\big)^2.
\]
The corresponding noise levels, discretisation parameters, and
LM settings
\(
(\alpha_0,\gamma,\varepsilon_J,N_{\max})
\)
used throughout the experiments are summarised in
Table~\ref{tab:parameters}.

To avoid inverse crimes, the synthetic observations are generated on a
finer computational mesh than that employed for inversion. The LM algorithm is initialised at \(u_0=\bar u=0\). At the discrete
level, the updates lie in the range of the discretised prior covariance,
which is the finite-dimensional analogue of the CM space. All computations were performed in MATLAB (R2024a) on a single shared-memory node (64 CPU cores, 128~GB RAM) on Ada, the University of Nottingham's HPC facility,  using multithreaded execution. 


\begin{table}[ht]
\centering
\caption{Numerical experiment parameters for the 1D and 2D test cases.}
\label{tab:parameters}
\renewcommand{\arraystretch}{1.2}
\begin{tabular}{lll}
\hline
\textbf{Parameter} & \textbf{1D Case} & \textbf{2D Case} \\
\hline
$\zeta_n$ 
& $0.2,\;0.4,\;0.6,\;0.8,\;1.0$
& $0.2,\;0.375,\;0.55,\;0.725,\;0.9$ \\

$T$
& $0.5$
& $0.45$ \\

Domain $D$
& $[0,1]$
& $[0,1]^2$\\

$(\mu_f,\phi,p_I,p_0)$
& $(1,1,2,1)$
& $(1,1,2,1)$\\

$N_s$
& $20$
& $25$ \\

$\chi$
& $0.01$
& $0.005$ \\

Resolution for inversion
& 250 cells
&3386 elements\\

Resolution for synthetic data
& 1000 cells
&5458 elements\\

$(\alpha_0,\gamma,\varepsilon_J,N_{\max})$
& $(10^4,0.2,0.01,50)$
& $(10^3,0.5,0.025,50)$\\

\hline
\end{tabular}
\end{table}

\subsection{1D results}
\label{sec:LMAP1D}

Figure~\ref{fig: Experiment1D} displays the true log-permeability field
\(u^\dagger\) together with the Gaussian posterior approximations
\(\{\mu_n\}_{n=1}^5\). For reference, the prior distribution
\(\mu_0\) is also shown before any observations are incorporated. For
each posterior approximation, we display the posterior mean
(i.e.\ the MAP estimator)
\(u_{\mathrm{MAP}}^{(n)}\),
together with the associated \(50\%\) and \(95\%\) pointwise credible
intervals.

The moving front locations
\(
\{\Upsilon[u^\dagger](t_n)\}_{n=1}^5
\)
corresponding to the true solution are indicated by the dashed red lines.
As additional measurements from later observation times are incorporated,
the posterior uncertainty decreases and the posterior mean progressively
approaches the true permeability field. Since pressure observations are
informative only within the saturated region, uncertainty reduction occurs
primarily behind the advancing front. Once the domain becomes fully
saturated and all observations have been assimilated, the LMAP posterior
approximation accurately captures the true permeability field within the
posterior credible intervals.

In all 1D experiments, the inversion converged in at most
14 forward solves and required less than \(1.5\) seconds of total
computation time.

\begin{figure}
    \centering
    \includegraphics[width=0.85\linewidth]{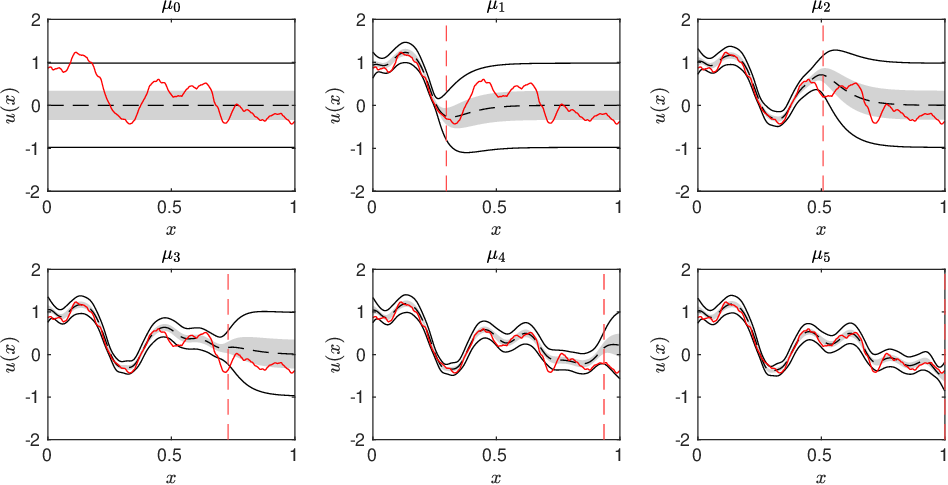}
    \caption{True log-permeability function, compared with sequence of posterior approximations to $\{\mu_i\}_{i=1}^5$ generated by LMAP.}
    \label{fig: Experiment1D}
\end{figure}

\subsubsection{Comparison with EKI and MCMC.} \label{subsec: LMAP1D_comparison}
To assess the computational efficiency and accuracy of the proposed LMAP
approximation, we compare the computation of the final posterior
approximation \(\mu_5\) from the previous experiment against both EKI and a reference posterior computed using
preconditioned Crank--Nicolson MCMC (pCN-MCMC).

The EKI implementation follows Algorithm~3 of \cite{MarcoYang}, using
ensemble sizes of \(500\), \(1000\), and \(5000\). As in \cite{GaussianApproximations}, the accuracy of each Gaussian
approximation is evaluated relative to a highly resolved reference
posterior obtained using pCN-MCMC \cite{pcnMCMC}. Specifically,
Algorithm~1 of \cite{GaussianApproximations} was run using
\(1.28\times10^6\) samples. The proposal scale was adapted during burn-in to obtain an acceptance
rate close to \(30\%\). It was then held fixed throughout the retained
sampling phase. After discarding the burn-in phase, a total of
\(
J_{\mathrm{pcn}}=10^6
\)
samples were retained.

The empirical posterior mean and standard deviation associated with the
pCN-MCMC samples
\(
\{u^{(j)}\}_{j=1}^{J_{\mathrm{pcn}}}
\)
are given by
\begin{align}
\bar u^{(\mathrm{pcn})}(x)
&=
\frac{1}{J_{\mathrm{pcn}}}
\sum_{j=1}^{J_{\mathrm{pcn}}}
u^{(j)}(x),
\\
\sigma^{(\mathrm{pcn})}(x)
&=
\sqrt{
\frac{1}{J_{\mathrm{pcn}}-1}
\sum_{j=1}^{J_{\mathrm{pcn}}}
\Big(
u^{(j)}(x)-\bar u^{(\mathrm{pcn})}(x)
\Big)^2
}.
\end{align}
These are treated as reference approximations of the posterior mean and
standard deviation associated with \(\mu_5\).

Let
\(
\bar u^{(\mathrm{LMAP})}
\)
and
\(
\bar u^{(\mathrm{EKI})}
\)
denote the posterior mean estimates obtained using LMAP and EKI,
respectively, with corresponding standard deviation estimates
\(
\sigma^{(\mathrm{LMAP})}
\)
and
\(
\sigma^{(\mathrm{EKI})}
\).
The quality of each approximation is assessed using the relative
\(L^2(D)\)-errors
\begin{align}
\mathcal E_{\mathrm{mean}}^{(\cdot)}
&=
\frac{
\|
\bar u^{(\mathrm{pcn})}
-
\bar u^{(\cdot)}
\|_{L^2(D)}
}{
\|
\bar u^{(\mathrm{pcn})}
\|_{L^2(D)}
},\quad
\mathcal E_{\mathrm{std}}^{(\cdot)}
=
\frac{
\|
\sigma^{(\mathrm{pcn})}
-
\sigma^{(\cdot)}
\|_{L^2(D)}
}{
\|
\sigma^{(\mathrm{pcn})}
\|_{L^2(D)}
}.
\end{align}

The parallelisable structure of both pCN-MCMC and EKI was exploited in
the numerical experiments. For pCN-MCMC, 64 independent Markov chains
were run simultaneously across 64 CPU cores. Each chain consisted of
\(2\times10^4\) samples, with the first \(4375\) discarded as burn-in,
and the remaining samples pooled together to form the final
\(10^6\) posterior samples. For EKI, the ensemble forward solves at each
iteration were distributed across the available CPU cores.

For the 1D LMAP experiments, parallelisation of the
representer computations was not exploited, since the representers were
obtained through explicit analytical evaluations rather than by solving
adjoint PDE systems. Nevertheless, in higher-dimensional settings where
representers are computed via independent adjoint solves, the framework
naturally admits substantial parallelisation across observations. While the computational cost of EKI and pCN-MCMC may be further reduced
through additional parallel processes, this necessarily comes at the
expense of increased computational resources. 

Figure~\ref{fig: Comparison1D} shows that the posterior distributions
produced by LMAP, EKI, and pCN-MCMC are qualitatively very similar. This
observation is further supported by
Table~\ref{tab: Comparison1D}, which compares the computational cost and
relative approximation errors of each method with respect to the
reference posterior. These results demonstrate that accurate Gaussian approximations of the
posterior distribution can be obtained without resorting to substantially
more expensive sampling-based methods. This is particularly important
for time-sensitive applications such as resin transfer moulding, where
rapid permeability estimation is essential for enabling real-time process
monitoring and active control. While both LMAP and EKI provide dramatic
reductions in computational cost relative to pCN-MCMC, the proposed LMAP
framework achieves accuracy comparable to large-ensemble EKI while
requiring approximately three orders of magnitude fewer forward model
evaluations. Note that for the 1D setting, the additional representer evaluations
required by LMAP correspond only to evaluations of the explicit formula
\eqref{eq:ri_simplified} and therefore contribute negligibly to the total
computational cost relative to the forward solves.

\begin{figure}
    \centering
    \includegraphics[width=0.9\linewidth]{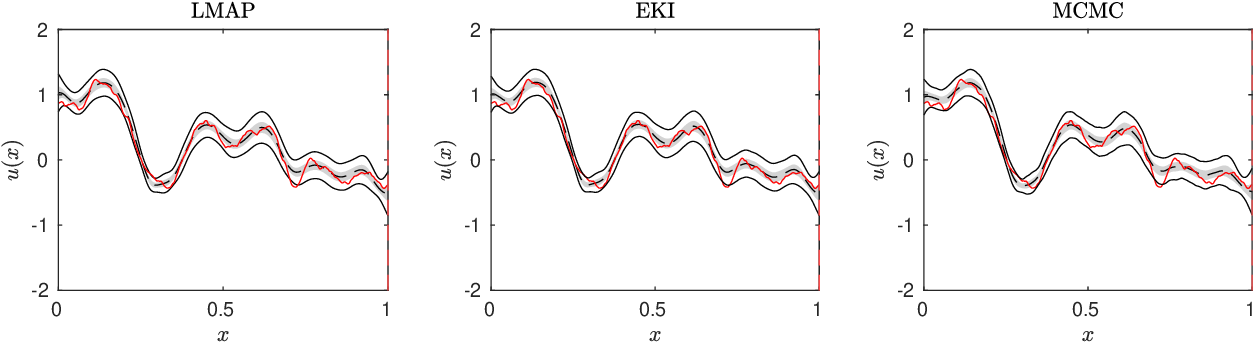}
    \caption{A comparison of each Gaussian approximation, computed using LMAP and EKI ($J_{\text{EKI}}=5000$) with the reference posterior approximation computed using pCN--MCMC.}
    \label{fig: Comparison1D}
\end{figure}

\begin{table}
\centering
\caption{Accuracy and speed of each Gaussian approximation, compared to MCMC.}
\begin{tabular}{ |c||c|c|c|c|c|c| } 
 \hline
 Algorithm 
 & \shortstack{$\mathcal{E}_{\text{mean}}^{(\cdot)}$\\ \%}
 & \shortstack{$\mathcal{E}_{\text{std}}^{(\cdot)}$\\ \%}
 & \shortstack{Forward\\ solves}
 & \shortstack{Representer\\ evaluations}
 & \shortstack{Multi-\\processing}
 & \shortstack{Computation\\ time}
 \\
 \hline
 LMAP & $\mathbf{3.43}$ & $\mathbf{5.11}$ & $\mathbf{1.30 \times 10^1}$ & $1.3\times10^3$ & \ding{56} & \textbf{1.43 secs}\\ 
 EKI (500) & 4.68 & 6.24 & $2.50\times 10^3$ & -- & \ding{51} & 15.2 secs\\
 EKI (1000) & 4.08 & 5.81 & $5.00\times 10^3$ & -- & \ding{51} & 28.9 secs\\
 EKI (5000) & 3.87 & 5.48 & $2.50\times 10^4$ & -- & \ding{51} & 2.34 mins\\
 pCN-MCMC & -- & -- & $1.28\times 10^6$ & -- & \ding{51} & 1.87 hrs\\
 \hline
\end{tabular}
\label{tab: Comparison1D}
\end{table}

\subsection{2D Results}
\label{sec: LMAP2D}

\begin{figure}
    \centering
    \includegraphics[width=0.95\textwidth]{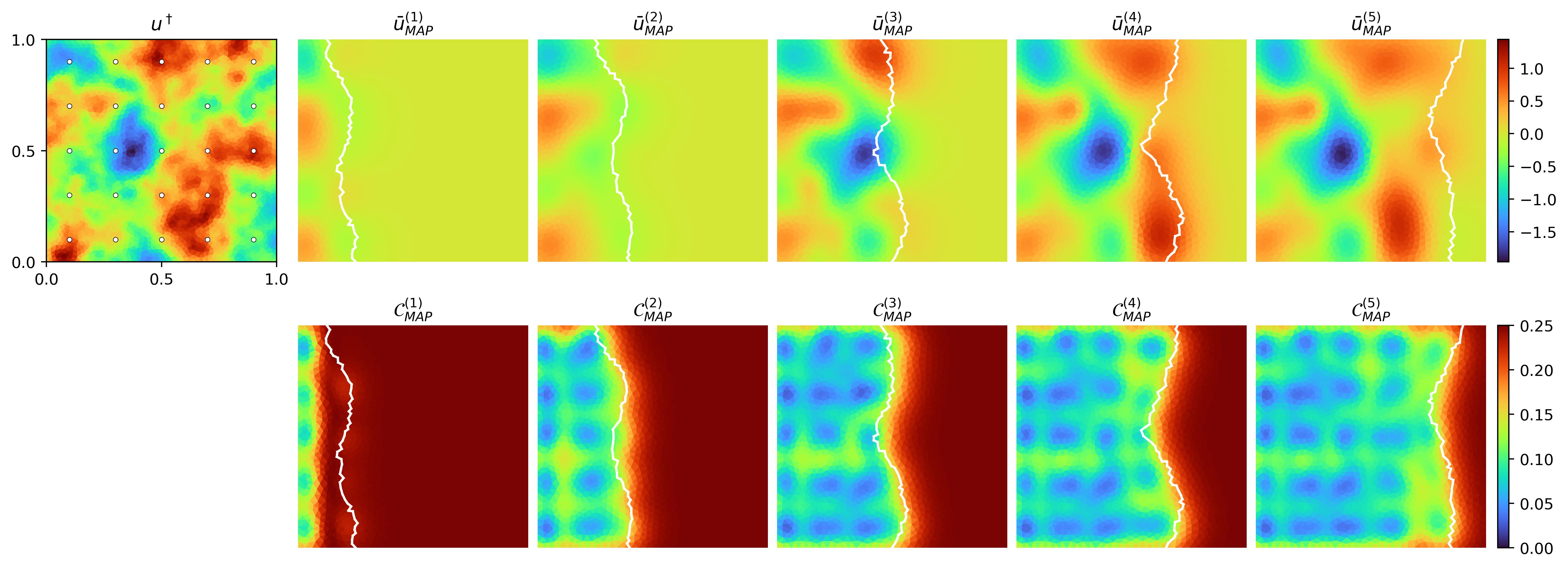}
    \caption{Sequence of LMAP approximations to \(u^\dagger\) using the 25 sensors (white). MAP estimates are shown in the top row, with the corresponding posterior variance fields shown below. The true front locations \(\{\Upsilon[u^\dagger](t_i)\}_{i=1}^5\) are overlaid in white.}
    \label{fig: sequential_LMAP2D}

    \includegraphics[clip,trim=0cm 0cm 0cm 0cm,width=0.39\linewidth]{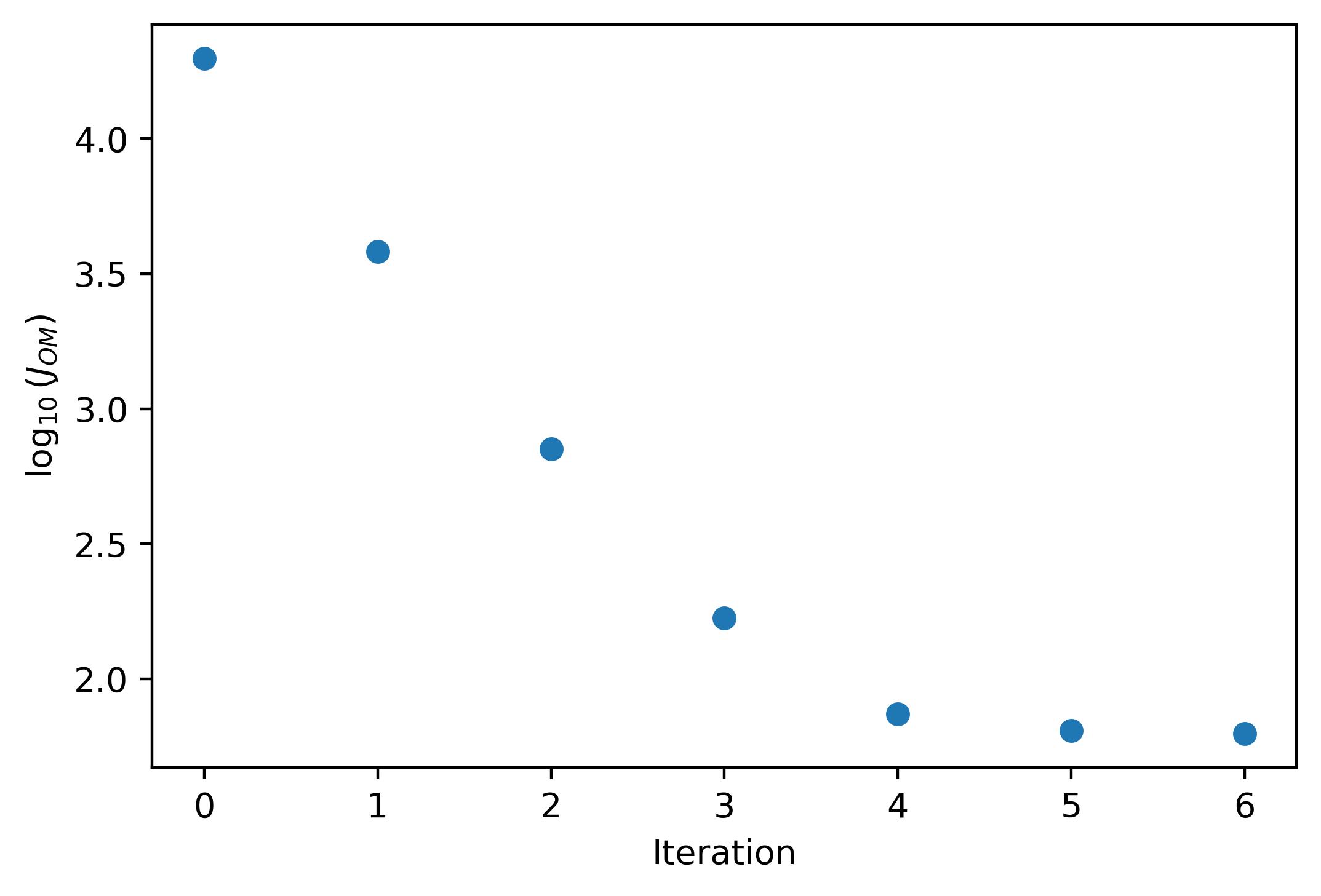}
    \caption{Onsager--Machlup cost as a function of the LM iteration number for the final-time posterior approximation.}
    \label{fig: OS_func}

    \includegraphics[clip,trim=4.5cm 1cm 2cm 0cm,width=0.95\linewidth]{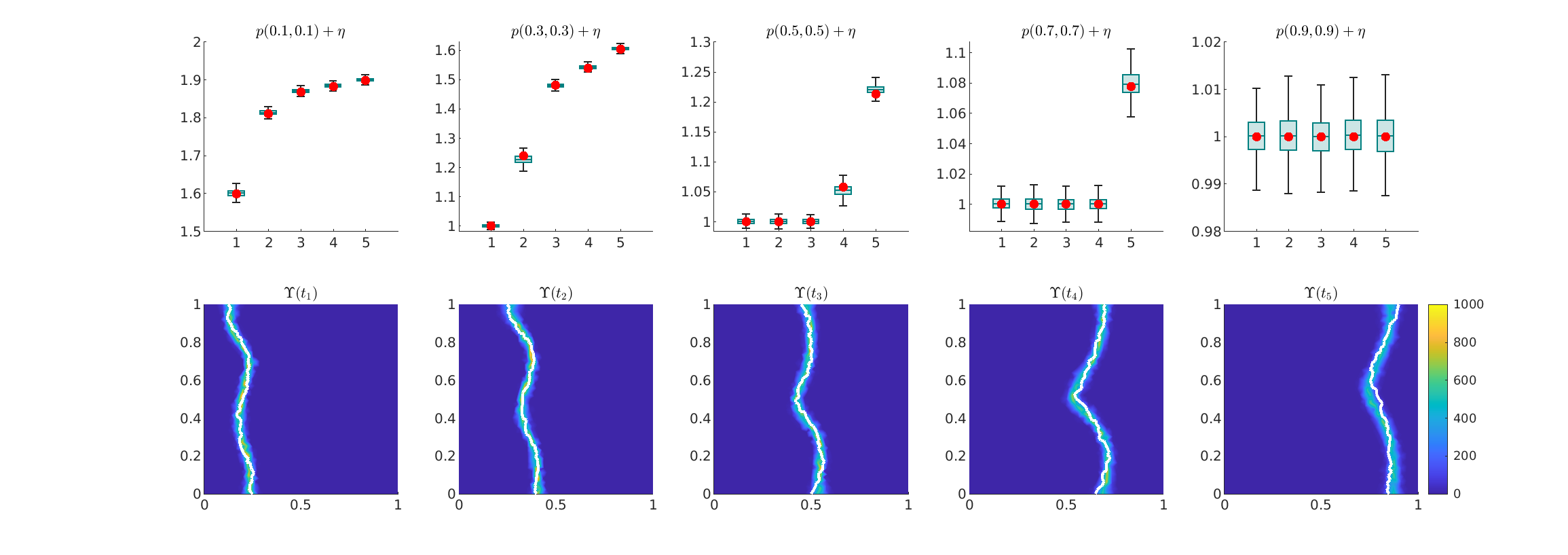}
    \caption{Top: noise-perturbed push-forward samples from the final-time LMAP approximation evaluated at selected sensors and compared against the observed data (red). Bottom: true front locations \(\{\Upsilon[u^\dagger](t_i)\}_{i=1}^5\) (white) compared with the empirical frequency of finite-element nodes classified as part of the moving boundary under posterior push-forward samples.}
    \label{fig: pushforward2D}
\end{figure}

Figure~\ref{fig: sequential_LMAP2D} shows the sequence of posterior mean
(MAP) and variance approximations together with the true
log-permeability field \(u^\dagger\). Within the saturated region at each
observation time, the posterior mean closely matches the true
permeability field, albeit with increased smoothness. This behaviour is
expected, since the MAP estimator belongs to the CM space
associated with the prior. The variance fields exhibit the expected
spatial structure: uncertainty is significantly reduced within the
saturated region and in the vicinity of the sensors, while remaining near
prior levels ahead of the advancing front where the permeability has not
yet influenced the flow dynamics.

The final-time MAP estimate was obtained after six LM iterations.
Figure~\ref{fig: OS_func} shows the corresponding decrease of the
Onsager--Machlup functional, with iteration zero corresponding to the
initial value \(J_{\mathrm{OM}}(\bar u)\). The computation times required
to obtain the successive posterior approximations corresponding to the
five observation times were
\[
(12.8,\ 20.1,\ 25.4,\ 59.7,\ 88.8)\ \mathrm{s}.
\]

To further assess the quality of the final-time Gaussian approximation,
samples were drawn from the corresponding LMAP posterior distribution and
propagated through the forward moving boundary model. The resulting
forward predictions were then perturbed with observational noise and
compared against the observed data. Simultaneously, the associated moving
boundary trajectories were recorded for each posterior sample.

Figure~\ref{fig: pushforward2D} shows that the observed data lie well
within the uncertainty intervals induced by the push-forward of the LMAP
posterior approximation. The empirical distribution of the
propagated moving fronts captures the true boundary locations
\(\{\Upsilon[u^\dagger](t_n)\}_{n=1}^5\) with good agreement. These
results indicate that the proposed Gaussian approximation provides a
meaningful local characterisation of posterior uncertainty not only in
parameter space, but also under nonlinear propagation through the moving
boundary dynamics.

As expected, increasing the number of spatial sensors leads to improved
reconstruction accuracy together with a corresponding reduction in
posterior uncertainty. The effect of sensor density on the quality of the
LMAP approximation is investigated further in
Appendix~\ref{subsec: SensorDensity}.

\subsubsection{Comparison with EKI.}

Compared with MCMC methods, which typically require
\(\mathcal O(10^5\text{--}10^6)\) forward solves, EKI has previously
demonstrated substantial computational savings in RTM inversion,
often requiring as few as \(\mathcal O(10^3)\) forward solves
\cite{Iglesias_2018,Matveev}. In this section, we demonstrate that the
proposed LMAP framework yields a further significant reduction in
computational cost.

For the 2D experiments, EKI is implemented using
Algorithm~3 of \cite{MarcoYang}. Ensemble sizes ranging from \(500\) to
\(5000\) were considered. Figure~\ref{fig: LMAP_vs_EKI} compares the
final-time EKI approximations, specifically the corresponding posterior
means and variances, with the LMAP approximation
\(
(u_{\mathrm{MAP}}^{(5)},\CC_{\mathrm{MAP}}^{(5)})
\)
shown in the final column of
Figure~\ref{fig: sequential_LMAP2D}. For small ensemble sizes, EKI
produces noticeably non-smooth mean estimates and provides a relatively
poor characterisation of uncertainty ahead of the moving front. As the
ensemble size increases, however, the EKI and LMAP approximations become
increasingly similar despite the fundamentally different nature of the
two methodologies.

We do not attempt to benchmark either method against a reference posterior,
since this would require computationally prohibitive MCMC simulations in 2D. Nevertheless,
Figure~\ref{fig: LMAP_vs_EKI} suggests that the empirical Gaussian approximations produced by
EKI become progressively closer to the LMAP approximation as the
ensemble size increases.

\begin{figure}[h]
    \centering
    \includegraphics[width=\linewidth]{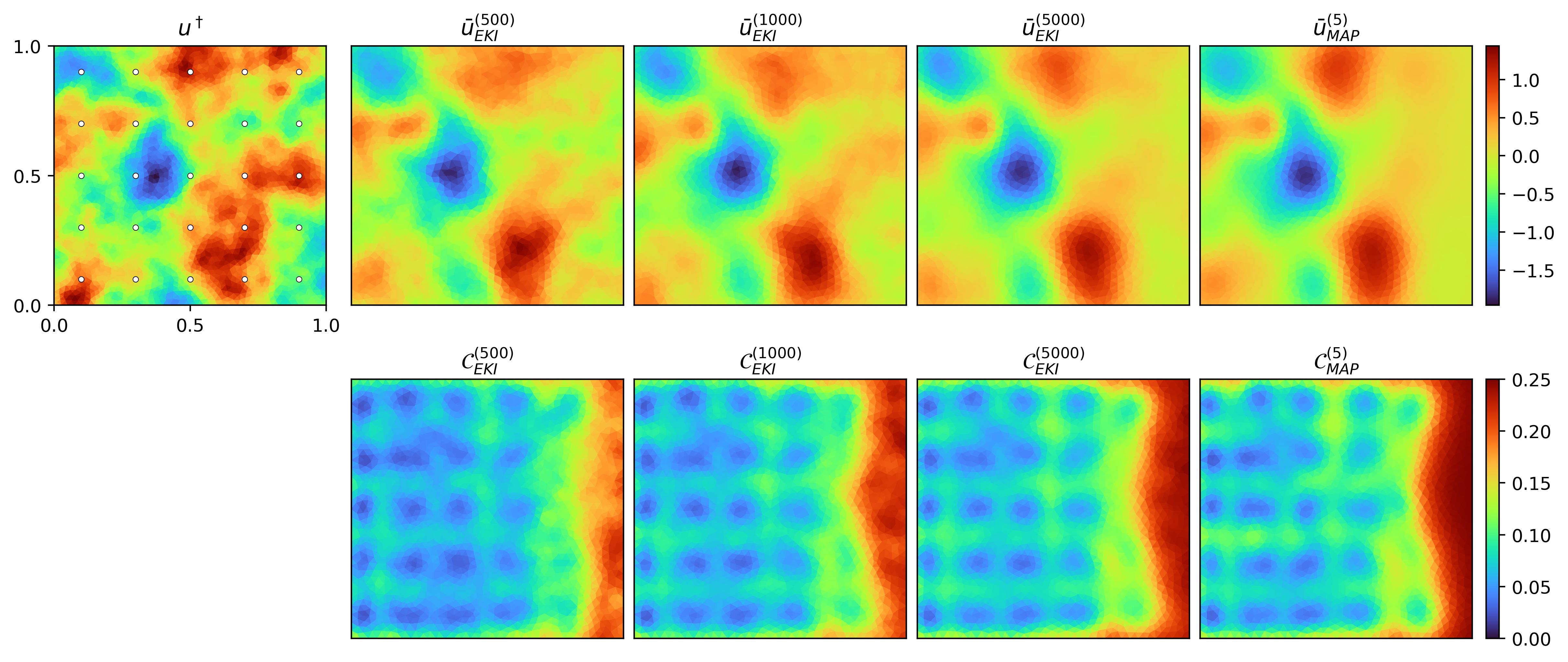}
    \caption{Comparison between the final-time EKI approximations and the LMAP approximation. The true permeability field \(u^\dagger\) is shown in the top-left panel together with the sensor configuration.}
    \label{fig: LMAP_vs_EKI}
\end{figure}

\begin{table}[t]
\centering
\caption{Comparison of computational cost for approximating \(\mu_5\).}
\begin{tabular}{|c||c|c|c|}
\hline
Algorithm
& \shortstack{Forward\\ solves}
& \shortstack{Adjoint\\ solves}
& \shortstack{Computation\\ time}
\\
\hline
LMAP
& $3.1\times10^1$
& $7.5\times10^2$
& 1.48 mins
\\
EKI (500)
& $3.0\times10^3$
& --
& 7.00 mins
\\
EKI (1000)
& $6.0\times10^3$
& --
& 14.3 mins
\\
EKI (5000)
& $3.0\times10^4$
& --
& 68.2 mins
\\
\hline
\end{tabular}
\label{tab: Comparison2D}
\end{table}

Table~\ref{tab: Comparison2D} highlights the substantial computational
advantage of the proposed LMAP framework over EKI. Even for relatively
small ensemble sizes, EKI requires at least two orders of magnitude more
forward solves than LMAP. For the largest ensemble size, which is
necessary to obtain posterior means and variances qualitatively
comparable to those produced by LMAP, the derivative-free EKI approach
requires more than one hour of computation time.

Unlike the 1D setting, the 2D LMAP framework
requires the solution of adjoint systems in order to construct the
representers. Of the 31 forward solves required by LMAP, one corresponds
to the initial evaluation of \(J_{\mathrm{OM}}(\bar u)\), while the
remaining solves arise from the parallel backtracking strategy employed
within the LM iterations. Since these backtracking evaluations are
performed simultaneously, the effective serial forward cost is
substantially smaller than the total number of forward solves reported in
Table~\ref{tab: Comparison2D}. The remaining computational cost is
primarily associated with adjoint solves, representer construction, and
the solution of the reduced linear systems. Nevertheless, even after
accounting for these additional adjoint computations, the overall runtime
remains substantially lower than that required by large-ensemble EKI.

\subsubsection{Application to complex 2D geometries.}\label{sec:complex_geometries}

We further demonstrate the proposed approach on more complex geometric
configurations: the annulus and the fork. The annulus configuration is
obtained by removing a circle of radius \(0.25\) from the interior of a
larger circle of radius \(1\), with both circles centred at the origin
(see Figure~\ref{fig: annulus}). The inner boundary serves as the inlet,
while the outer boundary acts as the outlet.

The fork configuration consists of a square region that branches into two
separate channels (see Figure~\ref{fig: fork}). The inlet
\(\partial D_I\) is located along the left edge of the domain, while the
outlet \(\partial D_0\) spans the ends of the two branches on the right.
All remaining boundaries are treated as no-flow boundaries and therefore
comprise \(\partial D_N\).

In the fork configuration, the moving front undergoes a topological
change as it splits into two disconnected components, thereby violating
the assumptions of the theoretical framework
(cf.\ Section~\ref{sec:admissible}). Nevertheless, the
numerical implementation remains robust and still produces meaningful
reconstructions and uncertainty estimates.

The experimental conditions for each configuration are summarised in
Table~\ref{tab: ExpConds}, notably differing from the previous example. Figures~\ref{fig: annulus} and
\ref{fig: fork} demonstrate the flexibility of the proposed LMAP
framework across distinct geometries and operating conditions. In
contrast to machine-learning-based surrogate approaches, which typically
require retraining when the geometry or physical setup changes, the same numerical inversion workflow can be applied without
reformulating the optimisation methodology. Moreover, whereas large-ensemble EKI approaches typically
require several hours of computation, the final-time LMAP estimates for
the annulus and fork configurations are obtained in \(248\) s and
\(136\) s, respectively.

\begin{figure}[h]
    \centering
    \includegraphics[width=\linewidth]{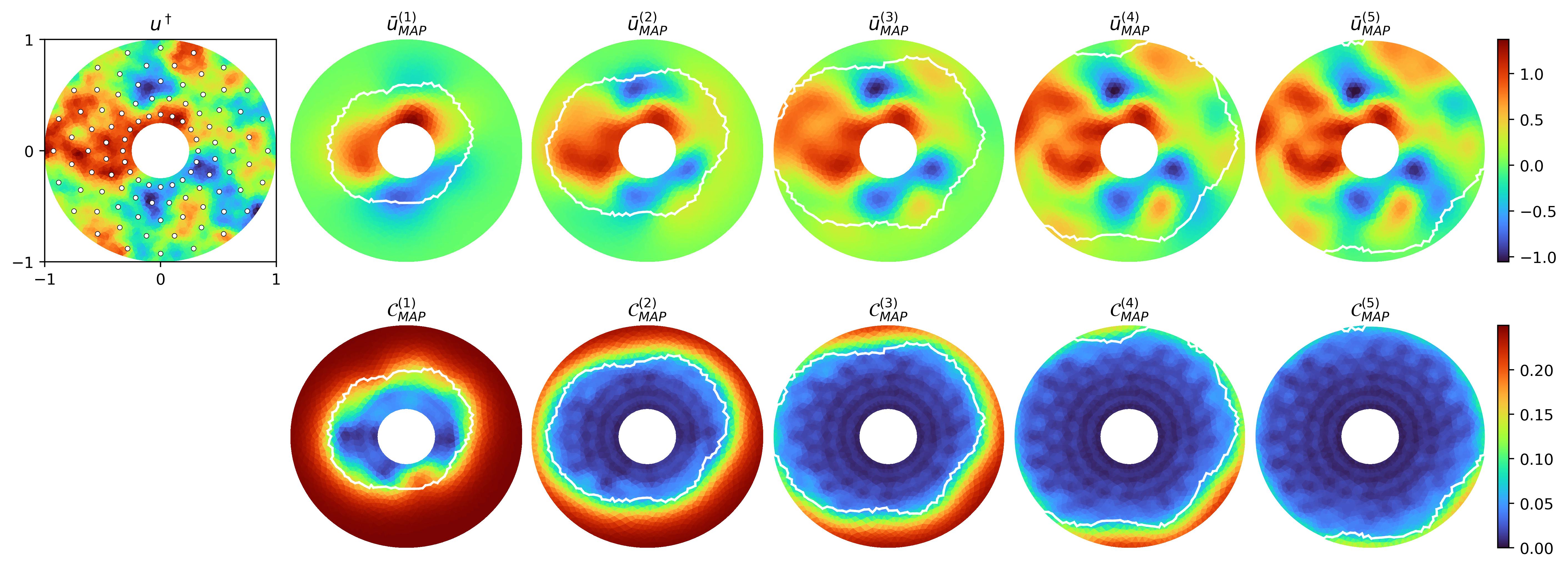}
    \caption{Experiment on annulus geometry. True log-permeability $u^\dagger$ shown in top left, along with the sensor configuration used. Top and bottom rows show the mean and variance of the LMAP estimate, respectively, at each observation time. True resin front location at each observation time $\{\Upsilon^\dagger(t_i)\}_{i=1}^5$ is overlaid in white for reference.}
    \label{fig: annulus}
    \includegraphics[width=\linewidth]{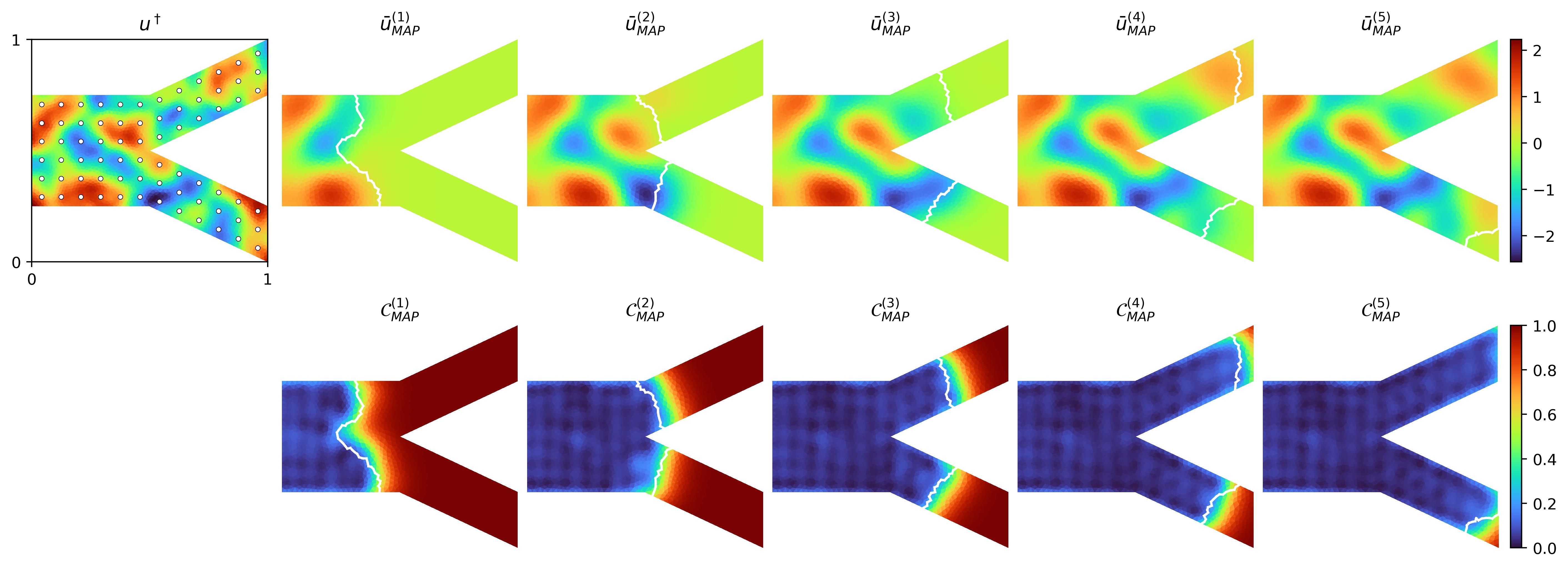}
    \caption{Experiment on the fork geometry. The true log-permeability
\(u^\dagger\) and sensor configuration are shown in the top-left panel.
The top and bottom rows show the LMAP posterior mean and variance,
respectively, at each observation time. The corresponding true moving
fronts are overlaid in white.}
    \label{fig: fork}
\end{figure}
\begin{table}
    \centering
    \caption{Setup for the fork and annulus experiments. `Elements' refers to the number of finite elements used within the forward solve for data generation (`forward') and the inversion via LMAP (`inverse').}
    \begin{tabular}{|c||c|c|c|c|}
     \hline
     Geometry & Elements (forward/inverse) & $(\mu_f,\phi,p_I,p_0)$ & $(\sigma,l,\nu)$\\
     \hline
     Annulus & $5572/3336$ & $(0.5, 0.5, 6, 1)$ & $(0.5,0.2,1.5)$\\
     Fork & $5439/3329$ & $(0.25, 0.75, 100, 0)$ & $(1,0.1,2.5)$\\
    \hline
    \end{tabular}
    \label{tab: ExpConds}
\end{table}

\section{Conclusion} \label{sec: Conclusion}

In this work, we introduced a variational framework for Bayesian
inversion in the Darcy-flow moving boundary problem describing resin
infusion in resin transfer moulding. The proposed approach combines a
very weak formulation of the coupled Darcy/interface evolution problem
with shape-calculus techniques to derive formal linearisations, adjoint
equations, and representer-based expressions for Fréchet derivatives of
the observation map. This enabled the construction of a reduced
Levenberg--Marquardt framework for MAP estimation together with local Gaussian posterior
approximations through the LMAP methodology.

The resulting framework provides an efficient alternative to
derivative-free inversion methods such as ensemble Kalman inversion and
sampling-based MCMC approaches. In 1D, where explicit
solutions are available, the proposed formulation was shown to be
consistent with the rigorous Fréchet differentiability theory of the
forward map. In 2D, numerical experiments
demonstrated that the proposed LMAP framework produces accurate
reconstructions together with meaningful uncertainty quantification while
requiring substantially fewer forward solves than large-ensemble EKI. The
results further showed that the framework remains effective even for
geometries exhibiting interface behaviours outside the assumptions of the
theoretical setting, including topological changes of the moving front.

A central aspect of the present work is that the substantial effort
required to derive and implement the linearised and adjoint moving
boundary systems from first principles ultimately leads to major
computational gains relative to derivative-free methodologies. Although
methods such as EKI are often attractive because they naturally provide
Gaussian uncertainty estimates without requiring sensitivities, the
results presented here indicate that the proposed LMAP approximation
produces posterior statistics comparable to those obtained from EKI and,
in the 1D setting, closely matches reference posteriors
computed using MCMC. At the same time, the proposed adjoint-based
framework achieves reductions of several orders of magnitude in the
number of forward solves and overall computational time. This reduction
is particularly significant in the context of RTM, where real-time or
near real-time permeability estimation is essential for online monitoring
and active process control.

We also note that the present variational framework does not preclude the use
of machine-learning methodologies. The forward solutions
and adjoint-based representers derived here provide the operator and
derivative information required for derivative-informed surrogate
modelling. One possible extension would therefore be to construct a
derivative-informed neural operator (DINO) approximation of the
parameter-to-observable map, trained jointly on evaluations of the map and
its Fr\'echet derivative
\cite{OLEARYROSEBERRY2022114199,OLEARYROSEBERRY2024112555}.
Such derivative-informed surrogates have been developed for Bayesian
optimal experimental design, PDE-constrained optimisation under
uncertainty, geometric MCMC, and amortised Bayesian inversion
\cite{doi:10.1137/23M157956X,10.1007/s10915-023-02145-1,JMLR:v26:24-0745,
JMLR:v27:25-0858}. 
In the present setting, these approaches could be used to accelerate
repeated evaluations of both the pressure observation map and its
representers within MAP estimation, posterior approximation, and
PDE-constrained control. Such a hybrid variational and
derivative-informed operator-learning framework may provide a practical
route towards real-time Bayesian estimation and control for industrial
RTM processes.

Beyond providing a convenient variational setting for the moving boundary
problem, the very weak formulation allows the geometric and PDE
perturbations to be decoupled in a manner compatible with adjoint-based
shape differentiation. In particular, the resulting structure enables the
linearised moving boundary conditions to emerge naturally from the
variational formulation after integration by parts, thereby providing a
systematic route for the derivation of adjoint representers and reduced
optimisation algorithms.

Several important analytical questions remain open. Existence of a weak solution to the coupled 2D Darcy/interface evolution problem follows from \cite{BLZ2018} (see  also references therein), under uniform ellipticity and boundedness of the permeability. A corresponding uniqueness result for spatially variable permeability on a general domain has not, to our knowledge, been established. In its absence, a rigorous well-posedness and differentiability theory for the coupled problem remains open. The adjoint and linearised systems derived here were obtained formally within the proposed variational framework, and placing them on a rigorous footing depends on such a theory. Developing it, together with convergence analysis of the numerical approximations, forms an important direction for future research.

From a computational perspective, several further extensions are also of
interest. In particular, the proposed framework naturally admits
parallelisation of the adjoint computations and could be combined with
adaptive sensor placement, online inversion strategies, or active-control
methodologies for RTM processes. Within a Bayesian framework, this also
connects naturally with optimal experimental design criteria, such as
A-optimal design, which seeks to minimise the trace of the posterior
covariance and thereby reduce parameter uncertainty
\cite{wright2023bayesian}. A natural next step is the extension of the
framework to fully three-dimensional RTM simulations, where the
computational advantages of the reduced adjoint-based methodology become
even more significant. 

Overall, the results indicate that variational and adjoint-based
approaches provide a computationally scalable framework for
Bayesian inversion in Darcy-flow moving boundary problems, particularly
in time-sensitive applications (e.g. active control) where repeated high-dimensional inference
is required.

\section*{Acknowledgments}
The authors thank Kristoffer G. van der Zee for stimulating discussions and for sharing his expertise on shape calculus.
This work was supported by the Engineering and Physical Sciences Research Council [grant number EP/P006701/1], through the EPSRC Future Composites Manufacturing Research Hub. We are also grateful for access to the University of Nottingham's Ada HPC service, where some of the simulations were performed. For the purpose of open access, the authors applied a Creative Commons Attribution (CC-BY) license to any Author Accepted Manuscript version arising.

\section*{Data availability}

No external datasets were used in this study. All numerical results are
based on synthetic data generated using the accompanying code. The code
and scripts required to reproduce the numerical experiments, figures and
tables are openly available at \url{https://github.com/ChaffyHurdle/RTM-flow/tree/lmap}. 

\appendix

\renewcommand{\thesection}{A}

\section{Proof of Theorem~\ref{thm:LM_LMAP_clean}} \label{app: lm_descent}
\begin{proof} Part (i): Let
\[
\Delta_k := y - \mathcal G(u_k) \in \mathbb R^M,
\qquad
D\mathcal G_k := D\mathcal G(u_k),
\qquad
g_k := \frac{\bar u - u_k}{1+\alpha_k} \in E.
\]
We write the LM increment as \(h = g_k + z\) for $z \in E$ and note that
\[
y - \mathcal G(u_k) - D\mathcal G(u_k) h
=(\Delta_k - D\mathcal G_k g_k) - D\mathcal G_kz.
\]
A direct computation shows that
\begin{align}
\frac12\|u_k + h - \bar u\|_E^2
+
\frac{\alpha_k}{2}\|h\|_E^2
=
\frac{1+\alpha_k}{2}\|z\|_E^2
+
\frac12\frac{\alpha_k}{1+\alpha_k}\|\bar u - u_k\|_E^2,
\end{align}
where the second term is independent of \(z\). Thus, minimising \( J_k(h)\) over \(h \in E\) is equivalent,
up to an additive constant, to minimising over \(z \in E\) the functional
\begin{equation}
J(z)
=
\frac12
\|(\Delta_k - D\mathcal G_k g_k) - D\mathcal G_k z\|_\Sigma^2
+
\frac{1+\alpha_k}{2}\|z\|_E^2.
\label{eq:Jz_clean}
\end{equation}
Define the CM representers
\begin{equation}
R_i[u_k] := \mathcal C_0 r_i[u_k] \in E.
\label{eq:E_rep}
\end{equation}
and notice that the characterisation of $D\GG_{k}$ given in \eqref{eq:DG_rep} can be written as
\begin{equation}
D\GG_{k} h=\Big[\langle R_1[u_k], h \rangle_{E},\dots, \langle R_M[u_k], h \rangle_{E}\Big]
\label{eq:DG_rep2}
\end{equation}
while 
\begin{equation}
[\mathcal R(u_k)]_{ij}=
\langle r_i[u_k], \mathcal C_0r_j[u_k] \rangle_{L^{2}(D)}=\langle R_i[u_k], R_j[u_k] \rangle_{E}.
\label{eq:mat_rep2}
\end{equation}
Let us now define
\[
S := \operatorname{span}\{R_1[u_k], \ldots, R_M[u_k]\} \subset E,
\qquad
S^\perp := \{ b \in E \mid \langle R_i[u_k], b \rangle_E = 0 \text{ for all } i=1,\ldots,M\}.
\]
and consider the orthogonal decomposition
\(
E = S \oplus S^\perp.
\)
Thus, any \(z \in E\) can be written uniquely as
\[
z = s + b,
\qquad s \in S,\; b \in S^\perp.
\]
Let us write \(s = \sum_{i=1}^{M}\beta_{i}R_{i}[u_k]\) for some \(\beta=(\beta_{1},\dots,\beta_{M}) \in \mathbb R^M\) (not necessarily unique). Hence from \eqref{eq:DG_rep2} and \eqref{eq:mat_rep2} we have
\[
D\mathcal G_k b = 0,
\qquad
D\GG_{k} s = \mathcal R(u_k)\beta.
\]
Moreover,
\[
\|z\|_E^2 = \|s\|_E^2 + \|b\|_E^2
= \beta^T \mathcal R(u_k)\beta + \|b\|_E^2.
\]
Substituting the above expressions into \eqref{eq:Jz_clean} yields
\begin{align}
J(\beta,b)
=
\frac12
\|(\Delta_k - D\GG_{k} g_k) - \mathcal R(u_k)\beta\|_\Sigma^2
+
\frac{1+\alpha_k}{2}\beta^T \mathcal R(u_k)\beta
+
\frac{1+\alpha_k}{2}\|b\|_E^2. \label{eq:apA1}
\end{align}
The minimisation in \(b\) is explicit: since the last term is strictly convex, the unique minimiser is $b_* = 0$. Thus the minimiser lies in \(S\), and we reduce to minimising over \(\beta\).

Differentiating \eqref{eq:apA1} with respect to \(\beta\) and setting the result equal to zero gives
\begin{equation}
-\mathcal R(u_k)\Sigma^{-1}
\big(\Delta_k - D\GG_{k} g_k - \mathcal R(u_k)\beta\big)
+
(1+\alpha_k)\mathcal R(u_k)\beta
= 0.
\label{eq:beta_opt_clean}
\end{equation}

Define
\begin{equation}
a_k :=
\frac{1}{1+\alpha_k}
\Sigma^{-1}
\big(\Delta_k - D\GG_{k} g_k - \mathcal R(u_k)\beta\big).
\end{equation}
Then \eqref{eq:beta_opt_clean} implies
\begin{equation}
\mathcal R(u_k)\beta = \mathcal R(u_k)a_k.
\label{eq:apA10}
\end{equation}
We now show that \(\beta\) and \(a_k\) define the same element in \(E\). Notice that
\[
\|\mathcal C_0D\GG_{k}^\ast (\beta - a_k)\|_E^2=(\beta-a_{k})^{T}D\GG_{k} \mathcal C_{0}D\GG_k^\ast(\beta-a_{k})=(\beta-a_{k})^{T}\mathcal R(u_k) (\beta-a_{k}) 
\]
and thus, from \eqref{eq:apA10}, $\mathcal C_0D\GG(u_{k})^\ast (\beta - a_k) = 0$, which using \eqref{eq:E_rep} and \eqref{eq:DG_rep_star} can be written as
\[
\mathcal C_0D\GG_{k}^\ast (\beta - a_k)
=
\sum_{i=1}^M (\beta_i - a_{k,i}) R_i[u_k] = 0,
\]
and arrive at
\[
\sum_{i=1}^M \beta_i R_i[u_k]
=
\sum_{i=1}^M a_{k,i} R_i[u_k].
\]
Thus, all minimising coefficient vectors correspond to the same element of
\(E\), and the unique minimiser is given by
\[
z_*  = \sum_{i=1}^M a_{k,i} R_i[u_k].
\]
Furthermore,
$$
D\GG_{k} z_* = \mathcal R(u_k)a_{k}= \mathcal R(u_k)\beta
$$
so
\begin{equation}
a_k :=
\frac{1}{1+\alpha_k}
\Sigma^{-1}
\big(\Delta_k - D\GG_{k} g_k - D\GG_{k} z_*\big).
\end{equation}
Using the previous two expressions combined with the change of variables \(h = g_k + z\) yields expressions \eqref{eq:MAP_update}-\eqref{eq:MAP_update2}.

\medskip

Part (ii). From \eqref{eq:DG_rep} we have that for any \(v\in L^2(D)\)
\begin{equation}
D\GG(u_{\mathrm{MAP}}) \mathcal C_{0}v=\Big[\langle r_1[u_{\mathrm{MAP}}], \mathcal C_{0} v \rangle_{L^{2}(D)},\dots, \langle r_M[u_{\mathrm{MAP}}], \mathcal C_{0} v \rangle_{L^{2}(D)}\Big]
\label{eq:DG_rep4}
\end{equation}
while 
\begin{equation}\label{eq:DG_rep5}
D\GG(u_{\mathrm{MAP}})\mathcal C_0 D\GG(u_{\mathrm{MAP}})^\ast=\mathcal R(u_{\mathrm{MAP}}),
\end{equation}
and 
\begin{equation}
\mathcal{C}_{0}D\GG(u_{\mathrm{MAP}})^{\ast} w = \sum_{i=1}^M w_i \mathcal{C}_{0}r_i[u_{\mathrm{MAP}}].
\label{eq:DG_rep6}
\end{equation}
for all $w\in \mathbb{R}^{M}$. Substituting \eqref{eq:DG_rep4}-\eqref{eq:DG_rep6} into \eqref{eq:Cov} completes the proof. 
\end{proof}

\renewcommand{\thesection}{B}

\section{Proof of Theorem~\ref{the:G_differentiable}} \label{app:G_differentiable}

\begin{lemma}[Local Lipschitz continuity of the extended pressure]
\label{lem:pressure_lipschitz}
Let \(D=(0,x^*)\), and let \(p[u]\) denote the extended pressure
defined by \eqref{eq: Analytical1D_end}. For every \(R>0\), there exists a
constant \(C_R>0\) such that
\[
\|p[u]-p[v]\|_{L^\infty(D_T)}
\leq
C_R\|u-v\|_X
\]
for all \(u,v\in X=C(\overline{D})\) satisfying
\[
\|u\|_X\leq R,
\qquad
\|v\|_X\leq R.
\]
\end{lemma}

\begin{proof}
Set
\[
\delta:=\|u-v\|_X,
\qquad
m_R:=e^{-R},
\qquad
M_R:=e^R.
\]
For \(u,v\) in the closed ball of radius \(R\) in \(X\), the mean
value theorem gives
\[
|e^{-u(x)}-e^{-v(x)}|
\leq
M_R\delta.
\]
Consequently, for every \(x\in[0,x^*]\),
\begin{align}
m_Rx
&\leq F[u](x)\leq M_Rx,
&
|F[u](x)-F[v](x)|
&\leq M_Rx\delta,
\label{eq:F_bounds}
\\
\frac{m_R}{2}x^2
&\leq W[u](x)\leq \frac{M_R}{2}x^2,
&
|W[u](x)-W[v](x)|
&\leq \frac{M_R}{2}x^2\delta.
\label{eq:W_bounds}
\end{align}

Fix \(t>0\), and write
\[
s:=\Upsilon[u](t),
\qquad
r:=\Upsilon[v](t).
\]
Without loss of generality, suppose that \(s\leq r\). If both fronts
are strictly smaller than \(x^*\), then
\[
W[u](s)=W[v](r)
=
\frac{p_I-p_0}{\mu_f\phi}t.
\]
Using \eqref{eq:W_bounds}, we obtain
\begin{align*}
\frac{m_R}{2}(r^2-s^2)\leq
W[v](r)-W[v](s)=
W[u](s)-W[v](s)\leq\frac{M_R}{2}s^2\delta.
\end{align*}
It follows that
\[
\frac{r-s}{r}
\leq
\frac{M_R}{m_R}\delta.
\]

If \(r=x^*\) while \(s<x^*\), then the \(v\)-front has already
reached the outlet, and
\[
W[v](x^*)
\leq
\frac{p_I-p_0}{\mu_f\phi}t
=
W[u](s)
<
W[u](x^*).
\]
Therefore,
\begin{align*}
\frac{m_R}{2}\bigl((x^*)^2-s^2\bigr)\leq
W[u](x^*)-W[u](s)\leq
W[u](x^*)-W[v](x^*)\leq
\frac{M_R}{2}(x^*)^2\delta,
\end{align*}
and hence the same type of estimate holds. Thus, in all cases,
\begin{equation}
\frac{|r-s|}{\max\{r,s\}}
\leq
C_R\delta.
\label{eq:front_lipschitz_relative}
\end{equation}

Let $\Delta p:=p_I-p_0$ and define the normalised extended pressure
\[
\pi_u(x,t)
:=
\frac{p[u](x,t)-p_0}{\Delta p}.
\]
For \(t>0\),
\[
\pi_u(x,t)
=
\begin{cases}
1-\displaystyle\frac{F[u](x)}{F[u](s)},
&0\leq x\leq s,\\[2mm]
0,
&s<x\leq x^*.
\end{cases}
\]

Again suppose that \(s\leq r\). For \(0\leq x\leq s\), we have
\begin{align*}
|\pi_u(x,t)-\pi_v(x,t)|\leq
\frac{|F[u](x)-F[v](x)|}{F[u](s)}
+
\frac{
F[v](x)\,
|F[v](r)-F[u](s)|
}{
F[u](s)F[v](r)
}.
\end{align*}
The first term is bounded by
\(
\frac{M_R}{m_R}\delta
\). Moreover,
\[
|F[v](r)-F[u](s)|
\leq
M_R(r-s)+M_Rs\delta.
\]
Using \eqref{eq:F_bounds},
\eqref{eq:front_lipschitz_relative}, and \(x\leq s\leq r\), it follows
that
\[
|\pi_u(x,t)-\pi_v(x,t)|
\leq
C_R\delta.
\]

For \(s<x\leq r\), we have \(\pi_u(x,t)=0\), while
\begin{align*}
0\leq\pi_v(x,t)=
1-\frac{F[v](x)}{F[v](r)}\leq
\frac{F[v](r)-F[v](s)}{F[v](r)}\leq
\frac{M_R}{m_R}\frac{r-s}{r}
\leq
C_R\delta.
\end{align*}
For \(x>r\), both extended pressures are equal to \(p_0\).
Consequently,
\[
\sup_{x\in D}
|\pi_u(x,t)-\pi_v(x,t)|
\leq
C_R\delta
\]
with a constant independent of \(t>0\). Multiplying by \(\Delta p\)
and taking the essential supremum over \(t\in(0,T)\) proves the result.
\end{proof}

\begin{proof}[Proof of Theorem~\ref{the:G_differentiable}]
Under the settings of Section \ref{sec:1D}, it is shown in \cite[Section 3.3]{MCthesis} that for a fixed $(x,t)\in D_{T}$, the mapping $u\to p[u](x,t)$ is Fréchet differentiable for all 
\begin{equation}\label{eq:appB0}
    u\in E_{\neq} := \{u \in E\, |\, \tau^*[u]\neq t ,\, \Upsilon[u](t) \neq x\}.
\end{equation}
Furthermore, it was derived that 
    \begin{align}\label{eq:appB_1}
        Dp[u;x,t]h = \langle r_p[u;x,t],h\rangle_{L^{2}(D)},\quad \forall h\in E, 
    \end{align}
where $r_p[u;x,t]\in L^{2}(D)$ is given by
\begin{align}
r_p[u;x,t](z)
&=
\mathbb{I}_{\{x<\Upsilon(t)\}}
\frac{p_I-p_0}{F[u](\Upsilon(t))^2}
\Bigg[
F[u](x)r_F[u;\Upsilon(t)](z)
-
F[u](\Upsilon(t))r_F[u;x](z)
\nonumber\\
&\qquad
-
\mathbb{I}_{\{t<\tau^*\}}
\frac{F[u](x)}{F[u](\Upsilon(t))}
e^{-u(\Upsilon(t))}
r_W[u;\Upsilon(t)](z)
\Bigg],
\label{eq: full_frechet_pressure}
\end{align}
and
\begin{align}
r_F[u;x](z)
:=
-\mathbb{I}_{\{z<x\}}e^{-u(z)},
 \,\,
r_W[u;x](z)
:=
-e^{-u(z)}\max\{x-z,0\}.
\label{eq: R_W}
\end{align}

In terms of the components of the forward map
\eqref{eq:observation_functionals}, define the linear operator
\begin{equation}
L_i[u]h
:=
\int_0^T\int_0^{x^*}
\mathcal H_i(x,t)Dp[u;x,t]h\,dx\,dt, 
\label{eq:appB_candidate_derivative}
\end{equation}
where $Dp[u;x,t]h$ is defined for almost every $(x,t)\in D_T$ as shown below together with
 \(L_i[u]=D\mathcal G_i(u)\).

Let us fix $u\in E$. From the pointwise differentiability result stated earlier, for fixed
$(x,t)\in D_{T}$ the map $u\mapsto p[u](x,t)$ is Fréchet differentiable at $u$
provided  $u\in E_{\neq}$, that is, $x\neq \Upsilon[u](t)$ and $t\neq \tau^*[u]$. For this fixed $u$, define
\[
B_u
:=
\{(x,t)\in D_T:x=\Upsilon[u](t)\}
\cup
\{(x,t)\in D_T:t=\tau^*[u]\}.
\]
The first set is the graph of the moving boundary and the second is a
horizontal line in space--time. Hence \(|B_u|=0\). It follows from the
pointwise differentiability result above that, for almost every
\((x,t)\in D_T\),
\[
\frac{
p[u+h](x,t)-p[u](x,t)-Dp[u;x,t]h
}{\|h\|_E}
\longrightarrow0
\qquad\text{as }\|h\|_E\longrightarrow0.
\]

By Lemma~\ref{lem:pressure_lipschitz} and the continuous embedding
\(E\hookrightarrow X\), there exist a neighbourhood of \(u\) in \(E\)
and a constant \(C_u>0\) such that
\[
\|p[u+h]-p[u]\|_{L^\infty(D_T)}
\leq
C_u\|h\|_E.
\]
Moreover, at every point at which the pointwise derivative exists, the
same local Lipschitz estimate implies
\[
|Dp[u;x,t]h|
\leq C_u\|h\|_E.
\]
Consequently,
\[
\frac{
\left|
p[u+h](x,t)-p[u](x,t)-Dp[u;x,t]h
\right|
}{\|h\|_E}
\leq 2C_u
\]
for almost every \((x,t)\in D_T\) and all sufficiently small
\(\|h\|_E\).

Since \(\mathcal H_i\in L^\infty(D_T)\) and \(D_T\) has finite measure,
the dominated convergence theorem yields
\begin{align*}
&\frac{
\left|
\mathcal G_i(u+h)-\mathcal G_i(u)-L_i[u]h
\right|
}{\|h\|_E}
\\
&\quad\leq
\int_0^T\int_0^{x^*}
|\mathcal H_i(x,t)|
\frac{
\left|
p[u+h](x,t)-p[u](x,t)-Dp[u;x,t]h
\right|
}{\|h\|_E}
\,dx\,dt
\longrightarrow0.
\end{align*}
Thus, \(\mathcal G_i\) is Fr\'echet differentiable at \(u\) and
\[
D\mathcal G_i(u)h=L_i[u]h.
\]

Using the \(L^2(D)\)-representer of the pressure derivative from
\eqref{eq:appB_1} and applying Fubini's theorem, we obtain
\begin{align*}
D\mathcal G_i(u)h
&=
\int_0^T\int_0^{x^*}
\mathcal H_i(x,t)
\int_0^{x^*}r_p[u;x,t](z)h(z)\,dz\,dx\,dt
\\
&=
\int_0^{x^*}
\left[
\int_0^T\int_0^{x^*}
\mathcal H_i(x,t)r_p[u;x,t](z)\,dx\,dt
\right]h(z)\,dz.
\end{align*}
Therefore,
\[
D\mathcal G_i(u)h
=
\langle r_i[u],h\rangle_{L^2(D)},
\]
where
\[
r_i[u](z)
=
\int_0^T\int_0^{x^*}
\mathcal H_i(x,t)r_p[u;x,t](z)\,dx\,dt.
\]
Substituting \eqref{eq: full_frechet_pressure} and
\eqref{eq: R_W} into this expression and simplifying gives
\eqref{eq:ri_simplified}.
\end{proof}

\renewcommand{\thesection}{C}
\section{Proof of Proposition~\ref{prop:LinearisedPhysics}} \label{app:LinearisedPhysics}
\begin{proof}
Let
\(\mathscr R(u,p,V,\Omega;\lambda,\kappa)\)
denote the left-hand side of \eqref{eq:weak1}. Throughout this proof, derivatives with respect to the domain are
understood as one-sided Eulerian directional derivatives at
\(s=0^+\), generated by the perturbation flow introduced in
Section~\ref{sec:geometric_perturbations}. For a shape-dependent
functional \(\mathscr F\), we write 
\[
d_\Omega^+\mathscr F(\Omega)[\deltao]
:=
\left.
\frac{d}{ds}
\mathscr F(\Omega_s)
\right|_{s=0^+},
\qquad
\Omega_s(t)=T_{s\deltao}(\Omega(t)).
\]
All identities involving time derivatives of the interface
parameterisation are understood for almost every \(t\in(0,T)\).

For fixed
\((u,V,\Omega;\lambda,\kappa)\), the dependence of
\(\mathscr R\) on \(p\) is linear and admits a continuous extension from
\(\mathcal P\) to the ambient space
\(\mathcal P_{\mathrm{amb}}\). Hence, for
\(\deltap\in\mathcal P_0(\Omega_T)\),
\begin{align*}
D_p\mathscr R(u,p,V,\Omega;\lambda,\kappa)[\deltap]
&=
-\int_0^T\int_{\Omega(t)}
\deltap\,\nabla\cdot(e^u\nabla\lambda)
\,dx\,dt
\\
&\quad
+
\int_0^T\int_{D\setminus\overline{\Omega(t)}}
\deltap\,\lambda
\,dx\,dt
\\
&=
-\int_0^T\int_{\Omega(t)}
\deltap\,\nabla\cdot(e^u\nabla\lambda)
\,dx\,dt,
\end{align*}
where the final equality follows from the zero exterior extension encoded
in \(\mathcal P_0(\Omega_T)\).

The variations with respect to \(u\) and \(V\) are obtained by direct
G\^ateaux differentiation. We therefore focus on the shape variation
with respect to the moving space--time domain. We split the residual as
\[
\mathscr R=\mathscr W+\mathscr M,
\]
where
\begin{align*}
\mathscr W(u,p,\Omega;\lambda)
&:=
-\int_0^T\int_{\Omega(t)}
p\,\nabla\cdot(e^u\nabla\lambda)\,dx\,dt
+
\int_0^T\int_{\partial\Omega(t)}
p_B e^u\nabla\lambda\cdot n\,ds\,dt
+\\
&\int_0^T\int_{D\setminus\overline{\Omega}(t)}
(p-p_0)\lambda\,dx\,dt,
\\
\mathscr M(V,\Omega;\lambda,\kappa)
&:=
\mu_f\phi
\int_0^T\int_{\Upsilon(t)}
\lambda V\,ds\,dt
+
\int_0^T\int_{\Upsilon(t)}
(W-V)\kappa\,ds\,dt .
\end{align*}


Integrating the boundary term in \(\mathscr W\) by parts yields
\begin{align*}
\mathscr W(u,p,\Omega;\lambda)
=
&-
\int_0^T\int_{\Omega(t)}
(p-p_B)\,
\nabla\cdot(e^u\nabla\lambda)
\,dx\,dt
\\
&+
\int_0^T\int_{\Omega(t)}
e^u\nabla p_B\cdot\nabla\lambda
\,dx\,dt +
\int_0^T\int_{D\setminus\overline{\Omega}(t)}
(p-p_0)\lambda
\,dx\,dt .
\end{align*}

Since \(\lambda=0\) on
\(\partial D_I\cup\partial D_0\),
\(\deltao\cdot n_D=0\) on \(\partial D_N\), and
\(p=p_0\) on \(\Upsilon(t)\), the shape derivative of the final term
vanishes by Proposition~3.2 of \cite{van2010goal}.

Moreover, by construction, \(p_B\) is chosen constant along normal
directions in a tubular neighbourhood of
\(\Upsilon(t)\cup\partial D_0^t\). Hence
\[
p-p_B=0
\qquad\text{on }\Upsilon(t)\cup\partial D_0^t,
\]
and the shape derivatives of the remaining volume integrals vanish by
Proposition~3.3 of \cite{van2010goal}. Consequently,
\[
d_\Omega^+\mathscr W(\Omega)[\deltao]=0,
\qquad
\forall\,\deltao\in\Theta_T.
\]

It remains to compute the one-sided shape variation of
\(\mathscr M\). By the spatial regularity assumed in the definition of the admissible
class, and using standard trace-extension results on \(C^{1,1}\)
interfaces, for almost every \(t\in(0,T)\) we may choose spatial
extensions
\[
V^E(\cdot,t),\;W^E(\cdot,t),\;\kappa^E(\cdot,t)\in H^2(D)
\]
to a tubular neighbourhood of \(\Upsilon(t)\), satisfying
\[
V^E=V,\qquad W^E=W,\qquad \kappa^E=\kappa
\quad\text{on }\Upsilon(t),
\]
and chosen constant along normal lines, so that
\[
\partial_n V^E=\partial_n W^E=\partial_n\kappa^E=0
\quad\text{on }\Upsilon(t).
\]
Applying the boundary shape differentiation formula of
Proposition~3.4 in \cite{van2010goal} gives
\begin{align}
d_\Omega^+\mathscr M(\Omega)[\deltao]
&=
\int_0^T\int_{\Upsilon(t)}
\partial_n\!\left(
(\mu_f\phi\lambda-\kappa^E)V^E
\right)\deltao_n
+
c(\mu_f\phi\lambda-\kappa^E)V^E\deltao_n
\,ds\,dt
\nonumber\\
&\quad
+
\int_0^T\int_{\Upsilon(t)}
\kappa^E\delta W
+
\left[
\partial_n(\kappa^E W^E)
+
c\kappa^E W^E
\right]\deltao_n
\,ds\,dt ,
\label{eq:N_shapeder_intermediate}
\end{align}
where \(\deltao_n:=\deltao\cdot n\), \(c=\operatorname{div}_{\Upsilon}n\),
and \(\delta W\) denotes the shape variation of the geometric normal
velocity.

Using the trace identities above and the weak kinematic constraint
\(W=V\) on \(\Upsilon(t)\), the terms involving \(\kappa\) cancel. Indeed,
on \(\Upsilon(t)\),
\[
\partial_n\!\left(
(\mu_f\phi\lambda-\kappa^E)V^E
\right)
+
\partial_n(\kappa^E W^E)
=
\mu_f\phi V\,\partial_n\lambda,
\]
and
\[
c(\mu_f\phi\lambda-\kappa)V
+
c\kappa W
=
c\mu_f\phi\lambda V.
\]
Therefore
\begin{align}
d_\Omega^+\mathscr M(\Omega)[\deltao]
&=
\mu_f\phi
\int_0^T\int_{\Upsilon(t)}
(\nabla\lambda\cdot n)V\,\deltao_n
\,ds\,dt
+
\mu_f\phi
\int_0^T\int_{\Upsilon(t)}
c\lambda V\,\deltao_n
\,ds\,dt
\nonumber\\
&\quad
+
\int_0^T\int_{\Upsilon(t)}
\kappa\,\delta W
\,ds\,dt .
\label{eq:N_shapeder1}
\end{align}

By the structure theorem for shape derivatives
\cite{delfour2011shapes}, the Eulerian shape derivative depends only on
the normal component of the perturbation field on the moving interface.
We therefore write
\[
\deltao_n:=\deltao\cdot n
\]
and, for the purpose of computing the first-order shape derivative,
replace the perturbation generator on \(\Upsilon(t)\) by its normal
representative
\[
\theta:=\deltao_n n,
\]
extended smoothly to a tubular neighbourhood of the interface. This
replacement does not change the Eulerian shape derivative.

Let \(T_{s\theta}\) denote the flow generated by this extension of
\(\theta\). Thus,
\[
\frac{\partial}{\partial s}T_{s\theta}(x,t)
=
\theta\bigl(T_{s\theta}(x,t),t\bigr),
\qquad
T_{0\theta}(x,t)=x.
\]
In particular,
\[
T_{s\theta}(x,t)
=
x+s\theta(x,t)+o(s)
\qquad\text{as }s\downarrow0.
\]

Tangential motion of the moving front changes only its
parameterisation and not its geometry. We therefore choose normal
parameterisations of the reference and perturbed interfaces for the
following calculation. Thus, for a tracked point
\(X(t)\in\Upsilon(t)\), we take
\[
\partial_tX(t)
=
W(X(t),t)n(X(t),t).
\]
Its perturbed counterpart is defined by
\begin{equation}
X_s(t)
:=
T_{s\theta}(X(t),t)
\in\Upsilon_s(t).
\label{eq:deltaV1}
\end{equation}
The perturbed point is likewise taken to satisfy
\begin{equation}
\partial_tX_s(t)
=
W_s(X_s(t),t)n_s(X_s(t),t),
\label{eq:deltaV3}
\end{equation}
where \(n_s\) is the outward unit normal to \(\Upsilon_s(t)\).

Using the first-order expansion of the perturbation flow gives, for
almost every \(t\in(0,T)\),
\begin{equation}
\partial_tX_s(t)
=
\partial_tX(t)
+
s\frac{D\theta}{Dt}(X(t),t)
+
o(s),
\label{eq:deltaV4}
\end{equation}
where \(D/Dt\) denotes the surface material derivative associated with
the chosen interface parameterisation.

Taking the one-sided derivative of
\eqref{eq:deltaV3} at \(s=0^+\), and using
\(\partial_tX=Wn\), yields
\[
\frac{D\theta}{Dt}
=
\delta W\,n+W\,\delta n,
\]
where
\[
\delta W
:=
\left.
\frac{d}{ds}
W_s(X_s(t),t)
\right|_{s=0^+},
\qquad
\delta n
:=
\left.
\frac{d}{ds}
n_s(X_s(t),t)
\right|_{s=0^+}.
\]
Taking the inner product with \(n\) gives
\[
\frac{D\theta}{Dt}\cdot n
=
\delta W+W\,\delta n\cdot n.
\]
Since \(n_s\cdot n_s=1\), one-sided differentiation at \(s=0^+\)
implies
\[
n\cdot\delta n=0.
\]
Consequently,
\[
\delta W
=
\frac{D\theta}{Dt}\cdot n.
\]

Using \(\theta=\deltao_n n\), we obtain
\[
\frac{D\theta}{Dt}\cdot n
=
\frac{D}{Dt}(\deltao_n n)\cdot n
=
\frac{D}{Dt}(\deltao_n)
+
\deltao_n\frac{Dn}{Dt}\cdot n.
\]
Since \(n\cdot n=1\),
\[
\frac{Dn}{Dt}\cdot n=0.
\]
Hence,
\[
\delta W
=
\frac{D}{Dt}(\deltao_n).
\]

Substituting this one-sided velocity variation into
\eqref{eq:N_shapeder1} gives
\begin{align*}
d_\Omega^+\mathscr M(\Omega)[\deltao]
&=
\mu_f\phi
\int_0^T\int_{\Upsilon(t)}
(\nabla\lambda\cdot n+c\lambda)V\,\deltao_n\,ds\,dt
+
\int_0^T\int_{\Upsilon(t)}
\kappa\,\frac{D}{Dt}(\deltao_n)\,ds\,dt .
\end{align*}

Combining the ordinary variations with respect to \(u\), \(p\), and
\(V\) with the one-sided Eulerian directional variation with respect
to \(\Omega\) gives the formal one-sided directional linearisation
\eqref{eq:linearised_weak}.

\end{proof}

Here the material derivative is understood in the surface sense along
the chosen interface parameterisation. On the moving-front portion, a
normal parameterisation is used in the calculation above, since
tangential interface motion changes only the parameterisation and not
the evolving geometry. Because
\[
t\longmapsto\varphi_t
\in
W^{1,1}\bigl(
0,T;C^{0,1}(\widehat\Upsilon;\mathbb R^2)
\bigr),
\]
the corresponding material derivatives are interpreted for almost every
\(t\in(0,T)\). Interface quantities are identified, when required, with
spatial extensions to a tubular neighbourhood of \(\Upsilon(t)\),
chosen constant along normal lines.

\renewcommand{\thesection}{D}

\section{Proof of Theorem~\ref{thm:adjoint_rep_2D}} \label{app: Representers2D}

\begin{proof}
Let $u\in E$ and let \(z=(p,V,\Omega_{T})\) be the solution to the very weak formulation of Definition~\ref{def:weak}. Let $h\in E$ and let $(\deltap,\delta V,\deltao)$ be the solution to the corresponding formal linearisation \eqref{eq:linearised_weak} at $(u,z)$. On the other hand, let $i\in \{1,\dots,M\}$ and $(\lambda_{i},\kappa_{i})$ be the solution to \eqref{eq:adjoint} and \(T^\ast\) defined via \eqref{eq:tstar}. Using $(\lambda,\kappa)=(\lambda_{i},\kappa_i)$ in \eqref{eq:linearised_weak} and rearranging yields:
\begin{align}
-\int_0^T\int_{\Omega(t)}
\deltap\,\nabla\cdot(e^u\nabla\lambda_{i})
\,dx\,dt 
+
\int_0^{T{^*}}\int_{\Upsilon(t)}
(\mu_f\phi\lambda_{i}-\kappa_{i})\delta V\,ds\,dt
\nonumber\\
\quad
+\int_0^{T^*}\int_{\Upsilon(t)}
\kappa_{i}\frac{D}{Dt}(\deltao_n)\,ds\,dt
+\mu_f\phi
\int_0^{T^*}\int_{\Upsilon(t)}
(\nabla\lambda_{i}\cdot n+c\lambda_{i})V\deltao_n\,ds\,dt\nonumber\\
=\int_0^T\int_{\Omega(t)}
p\,\nabla\cdot(e^u h\nabla\lambda_{i})
\,dx\,dt-\int_0^T\int_{\partial\Omega(t)}
p_B e^u h\nabla\lambda_{i}\cdot n
\,ds\,dt
\end{align}
Integrating by parts the first term in the right-hand side and using the definition of $p_{B}$ yields

\begin{align}
-\int_0^T\int_{\Omega(t)}
\deltap\,\nabla\cdot(e^u\nabla\lambda_{i})
\,dx\,dt 
+\int_0^{T{^*}}\int_{\Upsilon(t)}
(\mu_f\phi\lambda_{i}-\kappa_{i})\delta V\,ds\,dt\nonumber\\
+\int_0^{T{^*}}\int_{\Upsilon(t)}
\kappa_{i}\frac{D}{Dt}(\deltao_n)\,ds\,dt
+\mu_f\phi
\int_0^{T{^*}}\int_{\Upsilon(t)}
(\nabla\lambda_{i}\cdot n+c\lambda_{i})V\deltao_n\,ds\,dt\nonumber\\
=- \int_0^T\int_{\Omega(t)}
\,e^u h\,\nabla p\cdot \nabla\lambda_{i}
\,dx\,dt \label{eq:83}
\end{align}
By the transport theorem for moving curves \cite{gurtin1989transport}, we have
\[
\frac{d}{dt}
\int_{\Upsilon(t)}
\kappa_i\,\deltao_n\,ds
=
\int_{\Upsilon(t)}
\left[
\frac{D}{Dt}(\kappa_i\deltao_n)
+
cV\kappa_i\deltao_n
\right]\,ds .
\]
Expanding the material derivative gives
\[
\int_{\Upsilon(t)}
\kappa_i\frac{D}{Dt}(\deltao_n)\,ds
=
\frac{d}{dt}
\int_{\Upsilon(t)}
\kappa_i\deltao_n\,ds
-
\int_{\Upsilon(t)}
\left(
\frac{D\kappa_i}{Dt}
+
cV\kappa_i
\right)\deltao_n\,ds .
\]
Therefore,
\begin{align}
&\int_0^{T^\ast}\int_{\Upsilon(t)}
\kappa_i\frac{D}{Dt}(\deltao_n)\,ds\,dt
+
\mu_f\phi
\int_0^{T^\ast}\int_{\Upsilon(t)}
(\nabla\lambda_i\cdot n+c\lambda_i)V\deltao_n\,ds\,dt \notag
\\
&=
\left[
\int_{\Upsilon(t)}
\kappa_i\deltao_n\,ds
\right]_{0}^{T^\ast}
-
\int_0^{T^\ast}\int_{\Upsilon(t)}
\left[
\frac{D\kappa_i}{Dt}
+
cV(\kappa_i-\mu_f\phi\lambda_i)
-
\mu_f\phi V\nabla\lambda_i\cdot n
\right]\deltao_n\,ds\,dt . \label{eq:app5}
\end{align}
The admissible perturbations preserve the prescribed initial interface, so
\(\deltao_n(\cdot,0)=0\) on \(\Upsilon(0)\). Moreover, by construction, \(\kappa_i(\cdot,T^*(u))=0\). Hence the first term in the right-hand side of \eqref{eq:app5} vanishes while the left-hand side below coincides with that of \eqref{eq:adjoint} term by term.

Substituting \eqref{eq:app5} in \eqref{eq:83} then yields
\begin{align}
-\int_0^T\int_{\Omega(t)}
\deltap\,\nabla\cdot(e^u\nabla\lambda_{i})
\,dx\,dt 
+\int_0^{T{^*}}\int_{\Upsilon(t)}
(\mu_f\phi\lambda_{i}-\kappa_{i})\delta V\,ds\,dt\nonumber\\
-
\int_0^{T^\ast}\int_{\Upsilon(t)}
\left[
\frac{D\kappa_i}{Dt}
+
cV(\kappa_i-\mu_f\phi\lambda_i)
-
\mu_f\phi V\nabla\lambda_i\cdot n
\right]\deltao_n\,ds\,dt \notag
\\
=-\int_0^T\int_{\Omega(t)}
h e^u\nabla p\cdot\nabla\lambda_{i}\,dx\,dt . \label{eq:appD10}
\end{align}
On the other hand, $(\lambda_{i},\kappa_{i})$ satisfies \eqref{eq:adjoint} for all
$(\deltap,\delta V,\deltao) \in \mathcal P_0(\Omega_T)\times\mathcal V(\Omega_T)\times\Theta_T$
and, in particular, for the $(\deltap,\delta V,\deltao)$ given by the linearised weak form above. For this choice, the left-hand side of \eqref{eq:appD10} coincides with the left-hand side of \eqref{eq:adjoint}, yielding the desired result. 
\end{proof}

\renewcommand{\thesection}{E}
\section{Solving the adjoint equations} \label{app:adjoint_eqns}
We discuss the numerical implementation of \eqref{eq: Adjoint_i_1}-\eqref{eq: Adjoint_i_7}. We naturally separate this into a steady-state regime and a moving boundary regime.
\subsection*{Steady-state regime}
If \(\tau_k^*<T\), then for \(t\in[\tau_k^*,T]\) the domain is fully
saturated and the moving interface no longer evolves. In this case the
adjoint variable \(\lambda_i\) is obtained from the weak problem: find
\[
\lambda_i(\cdot,t)
\in
H^1_{0,\partial D_I\cup\partial D_0}(D)
:=
\{v\in H^1(D):v=0
\text{ on }\partial D_I\cup\partial D_0\}
\]
such that
\begin{align}
    \int_D
    e^{u_k}\nabla\lambda_i(\cdot,t)\cdot\nabla\psi\,dx
    =
    \int_D
    \mathcal H_i(\cdot,t)\psi\,dx,
    \qquad
    \forall\psi\in
    H^1_{0,\partial D_I\cup\partial D_0}(D).
    \label{eq: lambda_ss}
\end{align}
This problem is solved using standard finite elements with piecewise linear
basis functions.

\subsection*{Moving boundary regime}

For \(t<T^*\), the adjoint is solved on the saturated region
\(\Omega_k(t)\). Given \(\kappa_i(\cdot,t)\) on the moving interface, we seek
\[
\lambda_i(\cdot,t)
\in
\left\{
v\in H^1(\Omega_k(t)):
v=0
\text{ on }
\partial D_I\cup\partial D_0^t,
\quad
v=\frac{\kappa_i(\cdot,t)}{\mu_f\phi}
\text{ on }\Upsilon_k(t)
\right\}
\]
such that
\begin{align}
    \int_{\Omega_k(t)}
    e^{u_k}
    \nabla\lambda_i(\cdot,t)\cdot\nabla\psi\,dx
    =
    \int_{\Omega_k(t)}
    \mathcal H_i(\cdot,t)\psi\,dx,
    \label{eq: lambda_mb}
\end{align}
for all
\[
\psi\in
\left\{
v\in H^1(\Omega_k(t)):
v=0
\text{ on }
\partial D_I\cup\partial D_0^t\cup\Upsilon_k(t)
\right\}.
\]
Thus, once \(\kappa_i\) is known, \(\lambda_i\) is obtained from a standard
Dirichlet problem posed on the current saturated domain.
The terminal condition is imposed at the final active-front time $\kappa_i(\cdot,T^*)=0$.

Let \(\{t_j\}_{j=0}^{N_t}\) denote the time discretisation inherited from the
forward CVFEM solve, with \(t_0=0\) and
\(t_{N_t}\approx T^*\). Starting from
\(\kappa_i(\cdot,t_{N_t})=0\), we first solve
\eqref{eq: lambda_mb} at \(t_{N_t}\) to obtain
\(\lambda_i(\cdot,t_{N_t})\), and hence
\(\nabla\lambda_i(\cdot,t_{N_t})\). We then march backward in time using
\[
t_{j-1}=t_j-\Delta t_j,
\]
where \(\Delta t_j\) is inherited directly from the forward solve.

The evolution equation for \(\kappa_i\) is interpreted in the reference
parametrisation of the moving interface. Writing
\[
\widehat\kappa_i(X,t)
:=
\kappa_i(\varphi_t(X),t),
\qquad
X\in \widehat{\Upsilon}, 
\]
the surface material derivative satisfies
\[
\frac{D\kappa_i}{Dt}(\varphi_t(X),t)
=
\partial_t\widehat\kappa_i(X,t).
\]
Consequently, along tracked interface elements in the CVFEM discretisation,
the material derivative is approximated by a backward time difference.

Since each forward CVFEM step fills exactly one control volume \cite{MichaelMinho,Zenodo}, most of the
boundary nodes and elements comprising \(\Upsilon_k(t_j)\) persist in
\(\Upsilon_k(t_{j-1})\). This behaviour is illustrated schematically in
Figure~\ref{fig: kappa_mesh}.

\begin{figure}[h]
    \centering
    \includegraphics[width=0.3\linewidth]{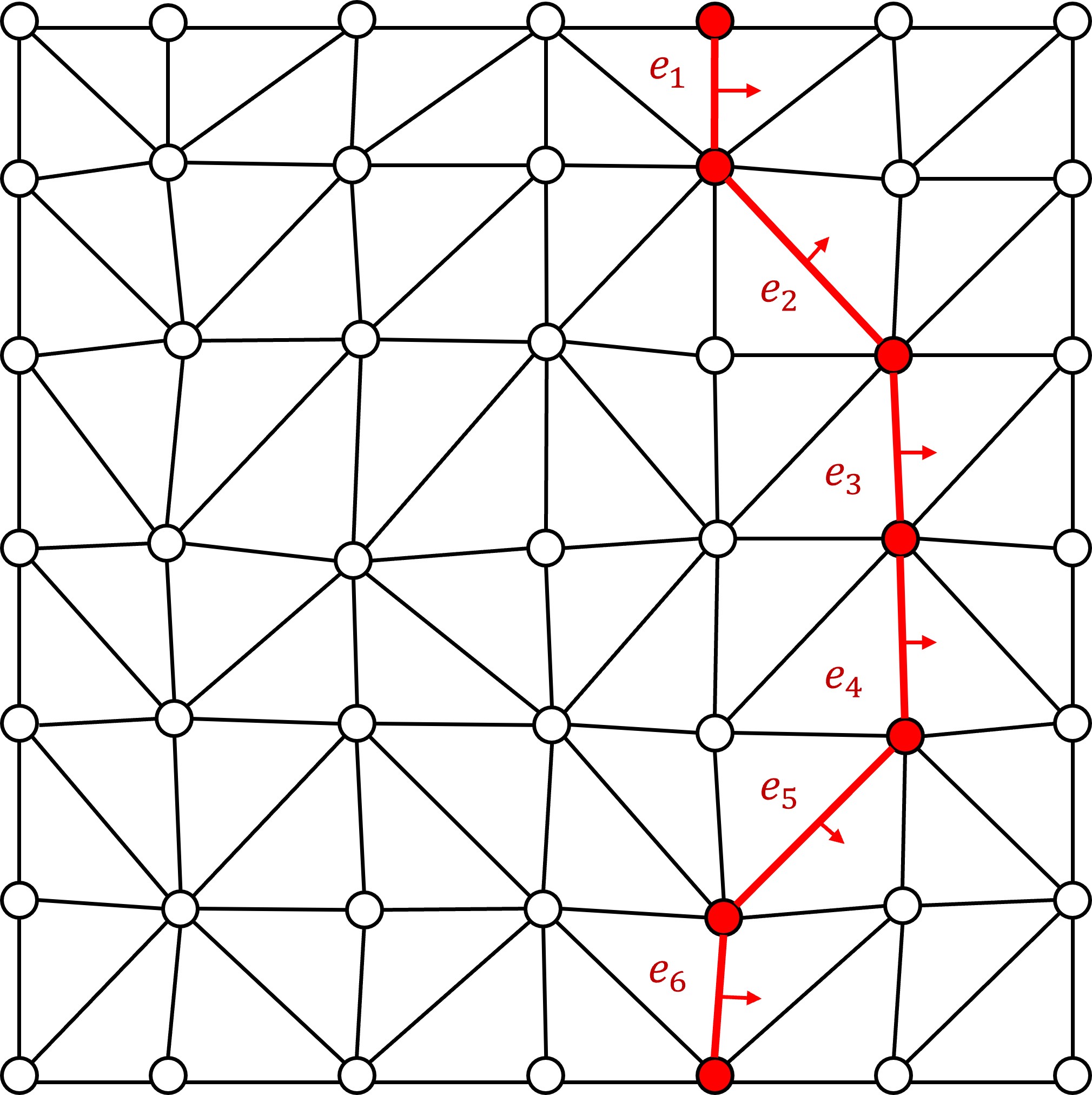}
    \quad\quad
    \includegraphics[width=0.3\linewidth]{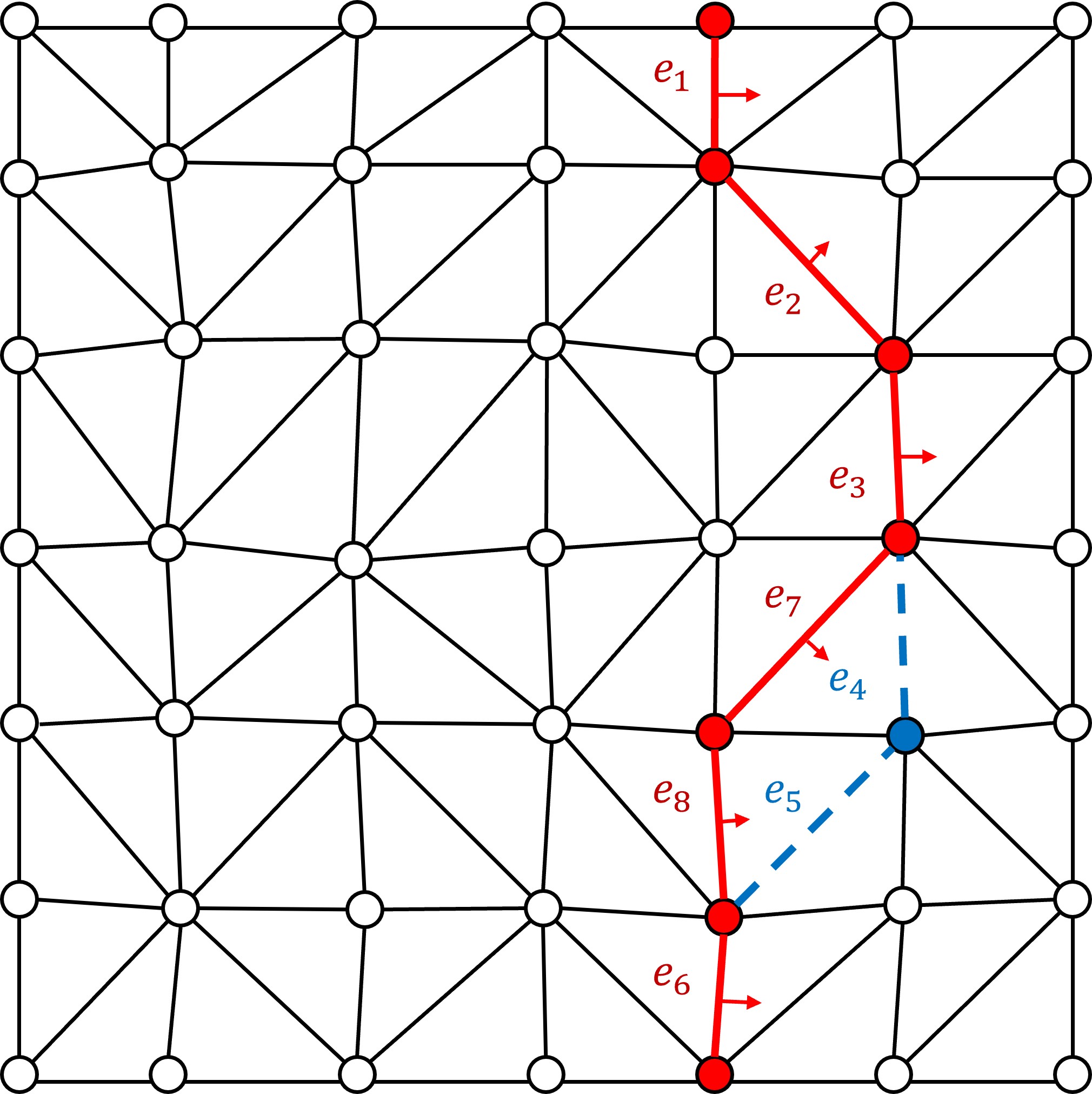}
    \caption{Active nodes, elements (denoted \(e\) in the figure), and edges
    comprising \(\Upsilon_k(t_j)\) (left) and
    \(\Upsilon_k(t_{j-1})\) (right) are shown in red. In blue are nodes,
    elements, and edges that were active at time \(t_j\) but inactive at
    \(t_{j-1}\). Also shown are the unit normals along each edge, used in the
    evaluation of the adjoint update.}
    \label{fig: kappa_mesh}
\end{figure}

For boundary elements that are active at both time levels, with centroid
\(x_e^c\), we update \(\kappa_i\) using the explicit reverse-time Euler
scheme
\begin{align}
    \kappa_i(x_e^c,t_{j-1})
    =
    \kappa_i(x_e^c,t_j)
    +
    \Delta t_j\,
    e^{u_k(x_e^c)}
    \big(\nabla p_k(x_e^c,t_j)\cdot n_e\big)
    \big(\nabla\lambda_i(x_e^c,t_j)\cdot n_e\big),
    \label{eq: exp_euler}
\end{align}
where \(n_e\) denotes the outward unit normal associated with the boundary
element.

For newly activated boundary elements at \(t_{j-1}\), the same update is
used, but the right-hand side is evaluated using a nearest-neighbour rule.
Specifically, the update is computed using the centroid of the closest
boundary element that was active at time \(t_j\).

The resulting elementwise values of
\(\kappa_i(\cdot,t_{j-1})\) are then averaged over all boundary elements
incident to each boundary node, producing a nodal approximation of
\(\kappa_i\) on \(\Upsilon_k(t_{j-1})\). These nodal values are imposed as
Dirichlet data in \eqref{eq: lambda_mb} to compute
\(\lambda_i(\cdot,t_{j-1})\). The procedure is repeated backward in time
until \(t=0\).

\renewcommand{\thesection}{F}
\section{Effect of sensor density} \label{subsec: SensorDensity}
A natural question, explored previously in the RTM setting \cite{Iglesias_2018,Causon2024}, is how sensor density affects the quality of the posterior approximation. To investigate this, we adopt the setup of Section~\ref{sec: LMAP2D}, but with a new true log-permeability field $u^\dagger \sim \NN(\ubar,\CC_0)$ and observational noise $\eta\sim \NN(0,\Sigma)$, yielding data $y= \GG(u^\dagger) + \eta$. The LMAP algorithm is then applied across the sensor grids: $3\times 3$, $5 \times 5$, $7\times 7$, $9\times 9$, and $15 \times 15$.

Figure~\ref{fig: sensor_density} shows the resulting posterior mean estimates at each observation time. Uncertainty is visualised via white masks with transparency $1 - (\sigma_i(x)/\sigma)^2$, where $\sigma_i$ denotes the posterior approximation to standard deviation at observation time $t_i$ and $\sigma$ is the prior standard deviation. Regions of low saturation (or `fog') indicate higher posterior variance. Increasing sensor density improves reconstruction accuracy and reduces uncertainty, with diminishing returns at higher sensor densities. The execution times for computing the final-time posterior approximation for sensor configurations $3\times 3$, $5\times 5$, $7\times 7$, $9\times 9$ and $15\times 15$ were $72$, $101$, $173$, $165$ and $363$ seconds, respectively.

\begin{figure}
    \centering
    \includegraphics[width=\linewidth]{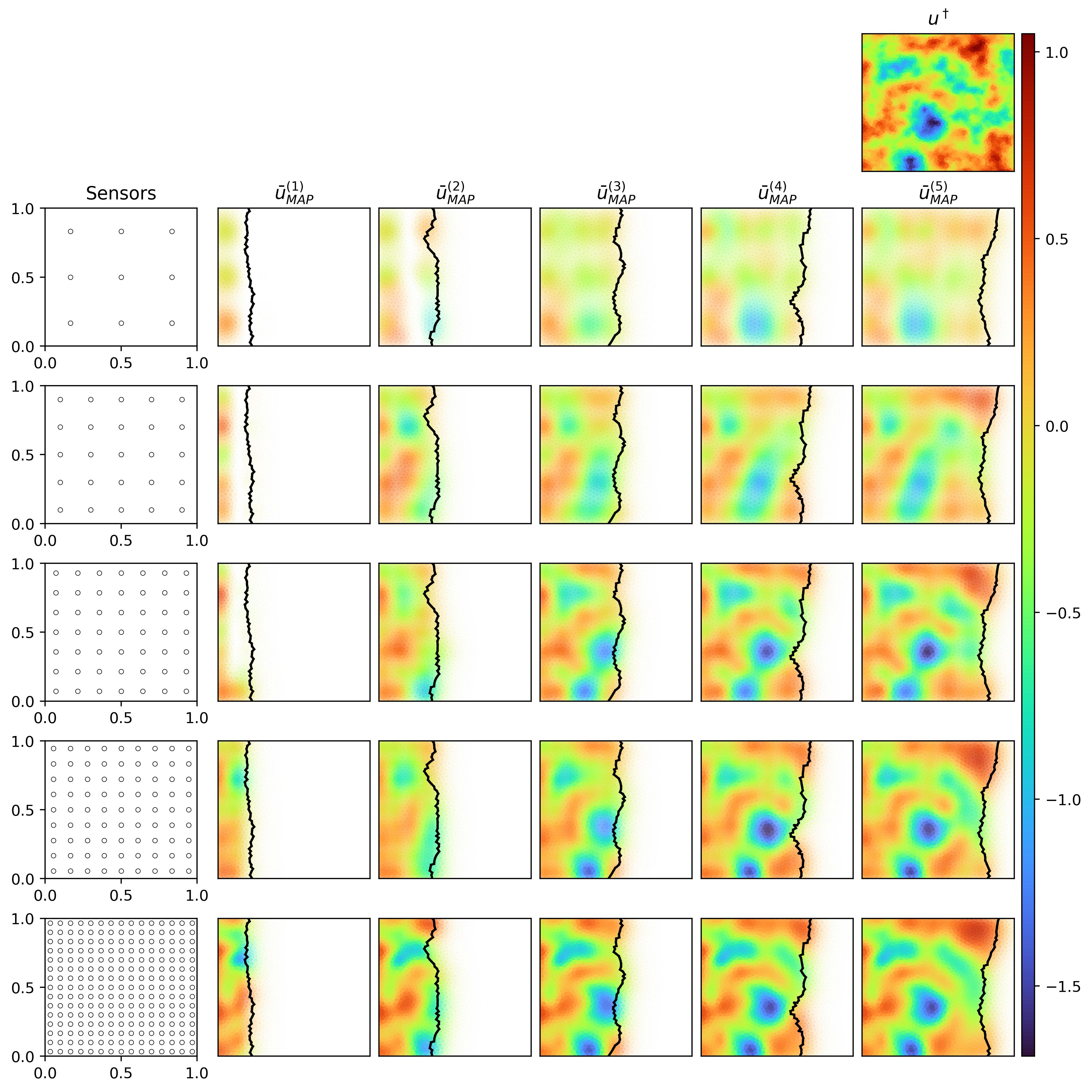}
    \caption{LMAP estimates for true log-permeability $u^\dagger$ field (top right) at each observation time, using various sensor densities ($3\times 3$, $5\times 5$, $7\times 7$, $9\times 9$ and $15\times 15$). True resin front location at each observation time $\{\Upsilon[u^\dagger](t_i)\}_{i=1}^5$ is overlaid for reference (black).}
    \label{fig: sensor_density}
\end{figure}



\bibliographystyle{iopart-num}
\bibliography{references}

\end{document}